\documentclass[a4paper,11pt]{article}

\pdfoutput=1 
\usepackage[utf8]{inputenc}
\usepackage[british]{babel}
\usepackage{amsthm}
\usepackage{amsmath}
\usepackage{amssymb}
\usepackage{tikz}
\usepackage{booktabs}
\usepackage{array}
\usepackage[font={small},width=0.9\textwidth]{caption}
\usepackage{tabularray}

\newtheorem{theorem}{Theorem}[section]
\newtheorem{corollary}[theorem]{Corollary}
\newtheorem{lemma}[theorem]{Lemma}

\theoremstyle{definition}
\newtheorem{example}[theorem]{Example}

\begin{document}

\begin{center}
{\Large\textbf{The Pattern Complexity of the 2-Dimensional \\[1ex] Paperfolding Sequence}}

\vspace{2ex}
Johan Nilsson
\end{center}

\begin{abstract}
	We present an exact formula for the number of distinct crease patterns in a square shaped region of a given size that appear in the 2 dimensional paperfolding structure.
\end{abstract}

\noindent{\small MSC2010 classification: 
05A15 Exact enumeration problems,
05B45 Tessellation and tiling problems, 
52C20 Tilings in 2 dimensions.}

\section{Introduction}

The classical paperfolding sequence is a well known sequence studied in many areas, such as combinatorics, computer science, and crystallography. This sequence is obtain by repeated folding of a long strip of paper in the middle, then unfolding it and reading of the different types of creases seen. Reading the creases from the left yields the paperfolding sequences, see \texttt{A014577} in \cite{oeis}. There are many versions and results around the paperfolding sequence, see \cite{allouch1995,allouch2023,dekking2012,dekking1982,gardner}, to mention some references. 

Ben-Abraham et al. \cite{ben} give a generalisation of the paperfolding sequence to higher dimensions via a recursive construction. Roughly descried, the 2-dimensional paperfolding sequence is obtained when a square shaped paper is folded in the middle, alternately between the axis parallel directions, and then reading of the crease structure seen when the paper is completely unfolded. In section~\ref{sec:prerequisites} we give a  detailed description of these folds, and state in detail the notions we use.

In this paper we focus on the complexity of the paperfoding structure in 2 dimensions. That is, we look at the number of crease patterns of a given size that occur in the paperfold structure. The complexity in the 1-dimensional case is given by Allouche in \cite{allouche1992}. The main result of this paper is the following theorem, which was conjectured in \cite{fgjn}. 

\begin{theorem}
	\label{thm:main}
	Let $A_n$ be the number of unique patterns of size $n \times n$ that occur in the 2-dimensional paperfolding sequence of infinte order. Then $A_1=4$, $A_2 = 68$, and 
	\begin{equation}
		\label{eq:AnSolution}
		A_n = 12\cdot n^2 + 24\cdot n \cdot 2^{\alpha} - 16\cdot 4^{\alpha} - 4,
	\end{equation}
	for $n \geq 3$, where $\alpha = \lfloor \log_2 (n-1)\rfloor$. 
\end{theorem}

The outline of this paper is as follows. In the next section we start by giving necessary definitions, and discuss how to generate the 2-dimensional paperfolding structure via substitution rules. Thereafter, in section~\ref{sec:patterns}, we turn to the question of defining what we mean by a \emph{pattern}, and look at the structure of sets of patterns. In the section thereafter we deduce a system of recursions describing the size of sets patterns. A solution to this system of recursion is presented in section~\ref{sec:proofofmainthm}, which then leads to the proof of Theorem~\ref{thm:main}.

\section{Prerequisites}
\label{sec:prerequisites}

In this section we present some background to the paperfoding structure in 2-dimensions, and we give some basic notation and fundamental definitions. 

The classical 1-dimensional paperfolding sequence is obtained, as briefly mention in the introduction, by repeatedly folding a long paper strip in the middle. By a fold we mean that we take the left end of the strip, and fold it over to the right, such that when unfolding there is a crease in the middle pointing downwards. When folding the strip over many times, and then when unfolding it completely we obtain a sequence of folds (sometimes called \emph{creases}), pointing upwards or downwards (also known as \emph{valley}- and \emph{mountain}-folds). The sequence of folds obtained is the so called paperfolding sequence.
 
In the 2-dimensional case, we swap the paper strip with an arbitrary large square paper. A fold in 2 dimensions is now made in two steps; first let the paper lay with its edges parallel to the normal coordinate system axis, with the paper's center at the origin. We start by taking the edge on the negative $x$-axis and fold in over to the positive $x$. Thereafter, the edge at the negative $y$ is folded over to the positive $y$. When unfolding the paper completely, we obtain a crease pattern on the paper, see Figure~\ref{fig:firstpaperfold}. As in the 1-dimensional case we make the assumption that we can fold the paper arbitrary many times before unfolding. 

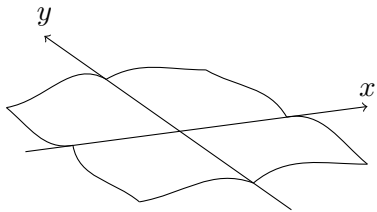
\begin{figure}[ht]
	\centering
	\begin{tikzpicture}[scale = 0.5, baseline={([yshift=-.8ex]current bounding box.center)}]
		\draw ( 4,0  ) to [out = 10, in = 160] ( 7, 0.5);
		\draw ( 7,0.5) to [out = 40, in = 180]  (10, 1);
		\draw (10,1)         to [out = 140, in = 10] (7.875, 2.25);
		\draw (7.875, 2.25) to [out = 120, in = 330]  (5.75,3.5);
		\draw (5.75,3.5)    to [out = 180, in =  30] (3.125, 3.25);
		\draw (3.125, 3.25) to [out = 150, in = 10]  (0.5,2.5);
		\draw (0.5,2.5) to      [out = 320, in = 190]  (2.25, 1.5);
		\draw (2.25, 1.5) to  [out = 280, in = 140]  (4,0);
		\draw[<-] (1.5, 4.4) node[above]{$y$} -- ( 8, -0.21);
		\draw[->] (1.0, 1.33)   -- (10.0, 2.53) node[above]{$x$};
	\end{tikzpicture}
	
	\caption{ The crease pattern obtained by folding a square paper, first from negative $x$ to positive $x$, followed by folding negative $y$ to positive $y$. When unfolding again, we obtain the crease pattern shown, with a mountain fold along the negative $x$, and valley folds
 along the other half-axis.}
 \label{fig:firstpaperfold}
\end{figure}

Let $\mathcal{S}_n$ denote the crease pattern obtained when folding (2-dimensional folds) a square paper $n$ times, and then unfolding it completely. We say that $\mathcal{S}_0$ is the unfolded paper. The crease pattern seen in Figure~\ref{fig:firstpaperfold} is the structure represented by $\mathcal{S}_1$. Let us denote the mountain folds with a filled dot and valley by circle, in the following way;
\begin{center}
	mountain
	\begin{tikzpicture}[scale = 0.5, baseline={([yshift=-.8ex]current bounding box.center)}]
		\draw (0,0) -- (0,1);
		\fill (0.0,0.5) circle (0.15);
	\end{tikzpicture}
	\qquad
	valley
	\begin{tikzpicture}[scale = 0.5, baseline={([yshift=-.8ex]current bounding box.center)}]
		\draw (0,0) -- (0,1);
		\draw (0.0,0.5) circle (0.15);
	\end{tikzpicture}
\end{center}

Let $\mathcal{S}_n^x$ denote the crease pattern of $\mathcal{S}_n$ reflected in the $x$-axis, (similar for $y$). Note that this operation swaps valley- and mountain folds. The crease pattern of $\mathcal{S}_n$ can now be described recursively via the structure given in  Figure~\ref{fig:recdefofSn}. This recursive definition was discussed in \cite{ben}. See Figure~\ref{fig:firstSnFold} for a visualization of the first $\mathcal{S}_n$s. 

\begin{figure}[ht]
	\centering
	\begin{tikzpicture}[scale = 0.5, baseline={([yshift=-.8ex]current bounding box.center)}]
		\draw[thin, color=gray, dotted] (-3,-3) rectangle (3,3);
		\draw (-3, 0) node[left]{$\mathcal{S}_{n} = $} -- (3,0);
		\draw (0, -3) -- (0, 3);
		\draw ( 1.5, 1.5) node{$\mathcal{S}_{n-1}$};
		\draw ( 1.5,-1.5) node{$\mathcal{S}_{n-1}^{x}$};
		\draw (-1.5, 1.5) node{$\mathcal{S}_{n-1}^{y}$};
		\draw (-1.5,-1.5) node{$\mathcal{S}_{n-1}^{xy}$};
		\fill (-1.5, 0.0 ) circle (0.15);
		\draw ( 1.5, 0.0) circle (0.15);
		\draw ( 0.0, 1.5) circle (0.15);
		\draw ( 0.0,-1.5) circle (0.15);
	\end{tikzpicture}
	\caption{The recursive structure of $\mathcal{S}_n$ for $n\geq1$, with $\mathcal{S}_0 = \emptyset$ (the unfolded paper).}
	\label{fig:recdefofSn}
\end{figure}

\begin{figure}[ht]
	\centering
	\begin{tabular}{cccc}
		$\emptyset$ 
		& \begin{tikzpicture}[scale = 0.5, baseline={([yshift=-.8ex]current bounding box.center)}]
	\draw[thin, color=gray, dotted] (0, 0) rectangle (2, 2);	
	\draw (0, 1) -- (2, 1);
	\draw (1, 0) -- (1, 2);
	\draw (1.0,1.5) circle (0.15);
	\fill (0.5,1.0) circle (0.15);
	\draw (1.5,1.0) circle (0.15);
	\draw (1.0,0.5) circle (0.15);
\end{tikzpicture}
		& \begin{tikzpicture}[scale = 0.5, baseline={([yshift=-.8ex]current bounding box.center)}]
	\draw[thin, color=gray, dotted] (0, 0) rectangle (4, 4);	
	\draw (0, 1) -- (4, 1);
	\draw (0, 2) -- (4, 2);
	\draw (0, 3) -- (4, 3);
	\draw (1, 0) -- (1, 4);
	\draw (2, 0) -- (2, 4);
	\draw (3, 0) -- (3, 4);
	\fill (1.0,3.5) circle (0.15);
	\draw (2.0,3.5) circle (0.15);
	\draw (3.0,3.5) circle (0.15);
	\fill (0.5,3.0) circle (0.15);
	\draw (1.5,3.0) circle (0.15);
	\fill (2.5,3.0) circle (0.15);
	\draw (3.5,3.0) circle (0.15);
	\fill (1.0,2.5) circle (0.15);
	\draw (2.0,2.5) circle (0.15);
	\draw (3.0,2.5) circle (0.15);
	\fill (0.5,2.0) circle (0.15);
	\fill (1.5,2.0) circle (0.15);
	\draw (2.5,2.0) circle (0.15);
	\draw (3.5,2.0) circle (0.15);
	\draw (1.0,1.5) circle (0.15);
	\draw (2.0,1.5) circle (0.15);
	\fill (3.0,1.5) circle (0.15);
	\draw (0.5,1.0) circle (0.15);
	\fill (1.5,1.0) circle (0.15);
	\draw (2.5,1.0) circle (0.15);
	\fill (3.5,1.0) circle (0.15);
	\draw (1.0,0.5) circle (0.15);
	\draw (2.0,0.5) circle (0.15);
	\fill (3.0,0.5) circle (0.15);
\end{tikzpicture}
		& \begin{tikzpicture}[scale = 0.5, baseline={([yshift=-.8ex]current bounding box.center)}]
	\clip( 0, -0.5) rectangle ( 8, 8);
	\draw[thin, color=gray, dotted] (0, 0) rectangle (8, 8);	
	\draw (0, 1) -- (8, 1);
	\draw (0, 2) -- (8, 2);
	\draw (0, 3) -- (8, 3);
	\draw (0, 4) -- (8, 4);
	\draw (0, 5) -- (8, 5);
	\draw (0, 6) -- (8, 6);
	\draw (0, 7) -- (8, 7);
	\draw (1, 0) -- (1, 8);
	\draw (2, 0) -- (2, 8);
	\draw (3, 0) -- (3, 8);
	\draw (4, 0) -- (4, 8);
	\draw (5, 0) -- (5, 8);
	\draw (6, 0) -- (6, 8);
	\draw (7, 0) -- (7, 8);
	\fill (1.0,7.5) circle (0.15);
	\fill (2.0,7.5) circle (0.15);
	\draw (3.0,7.5) circle (0.15);
	\draw (4.0,7.5) circle (0.15);
	\fill (5.0,7.5) circle (0.15);
	\draw (6.0,7.5) circle (0.15);
	\draw (7.0,7.5) circle (0.15);
	\fill (0.5,7.0) circle (0.15);
	\draw (1.5,7.0) circle (0.15);
	\fill (2.5,7.0) circle (0.15);
	\draw (3.5,7.0) circle (0.15);
	\fill (4.5,7.0) circle (0.15);
	\draw (5.5,7.0) circle (0.15);
	\fill (6.5,7.0) circle (0.15);
	\draw (7.5,7.0) circle (0.15);
	\fill (1.0,6.5) circle (0.15);
	\fill (2.0,6.5) circle (0.15);
	\draw (3.0,6.5) circle (0.15);
	\draw (4.0,6.5) circle (0.15);
	\fill (5.0,6.5) circle (0.15);
	\draw (6.0,6.5) circle (0.15);
	\draw (7.0,6.5) circle (0.15);
	\fill (0.5,6.0) circle (0.15);
	\fill (1.5,6.0) circle (0.15);
	\draw (2.5,6.0) circle (0.15);
	\draw (3.5,6.0) circle (0.15);
	\fill (4.5,6.0) circle (0.15);
	\fill (5.5,6.0) circle (0.15);
	\draw (6.5,6.0) circle (0.15);
	\draw (7.5,6.0) circle (0.15);
	\draw (1.0,5.5) circle (0.15);
	\fill (2.0,5.5) circle (0.15);
	\fill (3.0,5.5) circle (0.15);
	\draw (4.0,5.5) circle (0.15);
	\draw (5.0,5.5) circle (0.15);
	\draw (6.0,5.5) circle (0.15);
	\fill (7.0,5.5) circle (0.15);
	\draw (0.5,5.0) circle (0.15);
	\fill (1.5,5.0) circle (0.15);
	\draw (2.5,5.0) circle (0.15);
	\fill (3.5,5.0) circle (0.15);
	\draw (4.5,5.0) circle (0.15);
	\fill (5.5,5.0) circle (0.15);
	\draw (6.5,5.0) circle (0.15);
	\fill (7.5,5.0) circle (0.15);
	\draw (1.0,4.5) circle (0.15);
	\fill (2.0,4.5) circle (0.15);
	\fill (3.0,4.5) circle (0.15);
	\draw (4.0,4.5) circle (0.15);
	\draw (5.0,4.5) circle (0.15);
	\draw (6.0,4.5) circle (0.15);
	\fill (7.0,4.5) circle (0.15);
	\fill (0.5,4.0) circle (0.15);
	\fill (1.5,4.0) circle (0.15);
	\fill (2.5,4.0) circle (0.15);
	\fill (3.5,4.0) circle (0.15);
	\draw (4.5,4.0) circle (0.15);
	\draw (5.5,4.0) circle (0.15);
	\draw (6.5,4.0) circle (0.15);
	\draw (7.5,4.0) circle (0.15);
	\fill (1.0,3.5) circle (0.15);
	\draw (2.0,3.5) circle (0.15);
	\draw (3.0,3.5) circle (0.15);
	\draw (4.0,3.5) circle (0.15);
	\fill (5.0,3.5) circle (0.15);
	\fill (6.0,3.5) circle (0.15);
	\draw (7.0,3.5) circle (0.15);
	\fill (0.5,3.0) circle (0.15);
	\draw (1.5,3.0) circle (0.15);
	\fill (2.5,3.0) circle (0.15);
	\draw (3.5,3.0) circle (0.15);
	\fill (4.5,3.0) circle (0.15);
	\draw (5.5,3.0) circle (0.15);
	\fill (6.5,3.0) circle (0.15);
	\draw (7.5,3.0) circle (0.15);
	\fill (1.0,2.5) circle (0.15);
	\draw (2.0,2.5) circle (0.15);
	\draw (3.0,2.5) circle (0.15);
	\draw (4.0,2.5) circle (0.15);
	\fill (5.0,2.5) circle (0.15);
	\fill (6.0,2.5) circle (0.15);
	\draw (7.0,2.5) circle (0.15);
	\draw (0.5,2.0) circle (0.15);
	\draw (1.5,2.0) circle (0.15);
	\fill (2.5,2.0) circle (0.15);
	\fill (3.5,2.0) circle (0.15);
	\draw (4.5,2.0) circle (0.15);
	\draw (5.5,2.0) circle (0.15);
	\fill (6.5,2.0) circle (0.15);
	\fill (7.5,2.0) circle (0.15);
	\draw (1.0,1.5) circle (0.15);
	\draw (2.0,1.5) circle (0.15);
	\fill (3.0,1.5) circle (0.15);
	\draw (4.0,1.5) circle (0.15);
	\draw (5.0,1.5) circle (0.15);
	\fill (6.0,1.5) circle (0.15);
	\fill (7.0,1.5) circle (0.15);
	\draw (0.5,1.0) circle (0.15);
	\fill (1.5,1.0) circle (0.15);
	\draw (2.5,1.0) circle (0.15);
	\fill (3.5,1.0) circle (0.15);
	\draw (4.5,1.0) circle (0.15);
	\fill (5.5,1.0) circle (0.15);
	\draw (6.5,1.0) circle (0.15);
	\fill (7.5,1.0) circle (0.15);
	\draw (1.0,0.5) circle (0.15);
	\draw (2.0,0.5) circle (0.15);
	\fill (3.0,0.5) circle (0.15);
	\draw (4.0,0.5) circle (0.15);
	\draw (5.0,0.5) circle (0.15);
	\fill (6.0,0.5) circle (0.15);
	\fill (7.0,0.5) circle (0.15);
\end{tikzpicture} \\
		  $\mathcal{S}_0$
		& $\mathcal{S}_1$
		& $\mathcal{S}_2$
		& $\mathcal{S}_3$
	\end{tabular}
	\caption{The first crease patterns $\mathcal{S}_n$.}
	\label{fig:firstSnFold}
\end{figure}
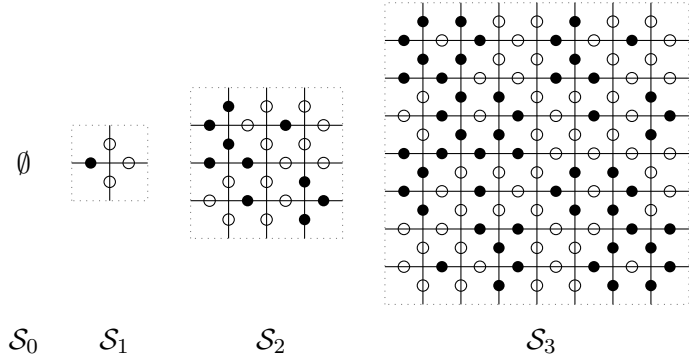

Our next step is to discuss how to generate the 2-dimensional paperfoding sequence via substitutions, as already considered in \cite{fgjn}. First, on the 16-letter alphabet $\mathcal{A} = \{\mathtt{A}, \mathtt{B},\mathtt{C},\ldots,\mathtt{P}\}$ we define the block substitution $\mu$ via 
\begin{equation}
\label{def:mu}    
\mu \,: \quad
\left\{
\text{\tt\small
\begin{tabular}{cccc}
	A $\mapsto$ \begin{tabular}{c} A F \\ G C \end{tabular}, & 
	B $\mapsto$ \begin{tabular}{c} A F \\ H D \end{tabular}, & 
	C $\mapsto$ \begin{tabular}{c} B E \\ G C \end{tabular}, & 
	D $\mapsto$ \begin{tabular}{c} B E \\ H D \end{tabular}, \\[3ex]
	E $\mapsto$ \begin{tabular}{c} A N \\ G K \end{tabular}, & 
	F $\mapsto$ \begin{tabular}{c} A N \\ H L \end{tabular}, & 
	G $\mapsto$ \begin{tabular}{c} B M \\ G K \end{tabular}, & 
	H $\mapsto$ \begin{tabular}{c} B M \\ H L \end{tabular}, \\[3ex]
	I $\mapsto$ \begin{tabular}{c} I F \\ O C \end{tabular}, & 
	J $\mapsto$ \begin{tabular}{c} I F \\ P D \end{tabular}, & 
	K $\mapsto$ \begin{tabular}{c} J E \\ O C \end{tabular}, & 
	L $\mapsto$ \begin{tabular}{c} J E \\ P D \end{tabular}, \\[3ex]
	M $\mapsto$ \begin{tabular}{c} I N \\ O K \end{tabular}, & 
	N $\mapsto$ \begin{tabular}{c} I N \\ P L \end{tabular}, & 
	O $\mapsto$ \begin{tabular}{c} J M \\ O K \end{tabular}, & 
	P $\mapsto$ \begin{tabular}{c} J M \\ P L \end{tabular}.
\end{tabular}}
\right.
\end{equation}
The second step, is to define the block substitution $\phi$ from $\mathcal{A}$ onto the 4-letter alphabet $\mathcal{B} = \{\mathtt{0}, \mathtt{1},\mathtt{2},\mathtt{3}\}$ via 
\begin{equation}
\label{def:phi}    
\phi \,: \quad
\left\{
\text{\tt\small
\begin{tabular}{cccc}
	A $\mapsto$ \begin{tabular}{c} 0 1 \\ 0 0 \end{tabular}, &
	B $\mapsto$ \begin{tabular}{c} 0 1 \\ 1 1 \end{tabular}, &
	C $\mapsto$ \begin{tabular}{c} 1 0 \\ 0 0 \end{tabular}, &
	D $\mapsto$ \begin{tabular}{c} 1 0 \\ 1 1 \end{tabular}, \\[3ex]
	E $\mapsto$ \begin{tabular}{c} 0 3 \\ 0 2 \end{tabular}, &
	F $\mapsto$ \begin{tabular}{c} 0 3 \\ 1 3 \end{tabular}, &
	G $\mapsto$ \begin{tabular}{c} 1 2 \\ 0 2 \end{tabular}, &
	H $\mapsto$ \begin{tabular}{c} 1 2 \\ 1 3 \end{tabular}, \\[3ex]
	I $\mapsto$ \begin{tabular}{c} 2 1 \\ 2 0 \end{tabular}, &
	J $\mapsto$ \begin{tabular}{c} 2 1 \\ 3 1 \end{tabular}, &
	K $\mapsto$ \begin{tabular}{c} 3 0 \\ 2 0 \end{tabular}, &
	L $\mapsto$ \begin{tabular}{c} 3 0 \\ 3 1 \end{tabular}, \\[3ex]
	M $\mapsto$ \begin{tabular}{c} 2 3 \\ 2 2 \end{tabular}, &
	N $\mapsto$ \begin{tabular}{c} 2 3 \\ 3 3 \end{tabular}, &
	O $\mapsto$ \begin{tabular}{c} 3 2 \\ 2 2 \end{tabular}, &
	P $\mapsto$ \begin{tabular}{c} 3 2 \\ 3 3 \end{tabular}. 
\end{tabular}
}
\right.
\end{equation}

An object of the form $\mu^n(x)$ where $x\in \mathcal{A}$ and $n \geq 0$ is called a \emph{supertile}. Here $\mu^n = \mu^{n-1}\circ \mu$ and $\mu^0 = Id$. Define the particular supertiles $T_n$ by $T_n := \mu^n(\mathtt{N})$, for $n\geq 0$. This definition can also be written as the block recursion 
\begin{equation}
	\label{eq:DefTnBlock}
	T_{n+1} =  \begin{array}{cc} \mu^n(\mathtt{I})  & T_n \\[1ex] \mu^n(\mathtt{P}) & \mu^n(\mathtt{L}) \end{array},
\end{equation}
and by $T$ we mean the limit of sequence of the $T_n$s. Based on $\mu$ and the $T_n$s, we let
\begin{equation}
	\label{eq:DefSnBlock}
	S_n := \phi( \mu^{n-1}(\mathtt{N}))  = \phi( T_{n-1})
\end{equation}
and $S = \phi( T )$. Note here that $T$ is made up of letters from the alphabet $\mathcal{A}$, while the letters in $S$ are from $\mathcal{B}$. See Figure~\ref{fig:firstSn} for an illustration of the initial $S_n$s. 

\begin{figure}[ht]
\centering
\begin{tabular}{cccc}
		{\small\tt\begin{tabular}{c} 2 3 \\ 3 3 \end{tabular}}
	& 	{\small\tt\begin{tabular}{c} 2 1 2 3 \\ 2 0 3 3 \\ 3 2 3 0 \\ 3 3 3 1 \end{tabular}}
	&	{\small\tt\begin{tabular}{c} 2 1 0 3 2 1 2 3 \\ 2 0 1 3 2 0 3 3 \\ 3 2 1 0 3 2 3 0 \\ 2 2 0 0 3 3 3 1 \\ 2 1 2 3 2 1 0 3 \\ 3 1 2 2 3 1 0 2 \\ 3 2 3 0 3 2 1 0 \\ 3 3 3 1 3 3 1 1 \end{tabular}} \\
		$S_1$
	&	$S_2$
	&	$S_3$
\end{tabular}
\caption{The first $S_i$ patterns.}
\label{fig:firstSn}
\end{figure}

Our final step to obtain the 2-dimensional paperfolding structure via a substitution rule is to decorate each letter in $\mathcal{B}$ with a crease pattern according to the following rule
\[
	\begin{tabular}{cccc}
{\tt 0} $\mapsto$
\begin{tikzpicture}[scale = 0.5, baseline={([yshift=-.8ex]current bounding box.center)}]
	\draw[cap = rect] (0, 1) -- (0, 0) --  (1, 0);
	\draw[thin, color=gray, dotted] (1, 0) -- (1, 1) -- (0, 1);
	\fill (0.0,0.5) circle (0.15);
	\fill (0.5,0.0) circle (0.15);
\end{tikzpicture}\, ,
&
{\tt 1} $\mapsto$
\begin{tikzpicture}[scale = 0.5, baseline={([yshift=-.8ex]current bounding box.center)}]
	\draw[cap = rect] (0, 1) -- (0, 0) --  (1, 0);
	\draw[thin, color=gray, dotted] (1, 0) -- (1, 1) -- (0, 1);
	\fill (0.0,0.5) circle (0.15);
	\draw (0.5,0.0) circle (0.15);
\end{tikzpicture}\, ,
&
{\tt 2} $\mapsto$
\begin{tikzpicture}[scale = 0.5, baseline={([yshift=-.8ex]current bounding box.center)}]
	\draw[cap = rect] (0, 1) -- (0, 0) --  (1, 0);
	\draw[thin, color=gray, dotted] (1, 0) -- (1, 1) -- (0, 1);
	\draw (0.0,0.5) circle (0.15);
	\fill (0.5,0.0) circle (0.15);
\end{tikzpicture}\, ,
&
{\tt 3} $\mapsto$
\begin{tikzpicture}[scale = 0.5, baseline={([yshift=-.8ex]current bounding box.center)}]
	\draw[cap = rect] (0, 1) -- (0, 0) --  (1, 0);
	\draw[thin, color=gray, dotted] (1, 0) -- (1, 1) -- (0, 1);
	\draw (0.0,0.5) circle (0.15);
	\draw (0.5,0.0) circle (0.15);
\end{tikzpicture}\, .
\end{tabular}
\]
In Figure~\ref{fig:firstSndec}, the decorations used on the initial $S_n$s are given. 

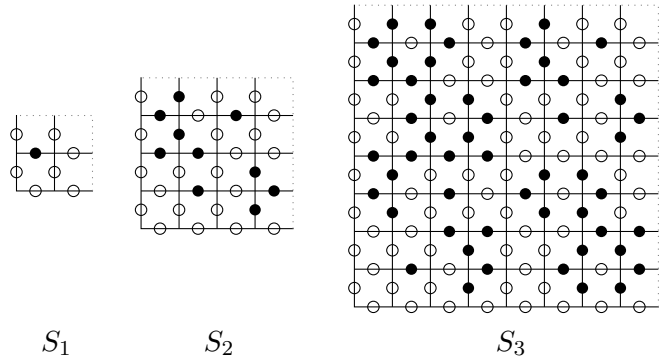
\begin{figure}[ht]
\centering
\begin{tabular}{cccc}
		\begin{tikzpicture}[scale = 0.5, baseline={([yshift=-.8ex]current bounding box.center)}]
	\draw[cap = rect] (0, 2) -- (0, 0) --  (2, 0);
	\draw[thin, color=gray, dotted] (2, 0) -- (2, 2) -- (0, 2);
	\draw (0, 1) -- (2, 1);
	\draw (1, 0) -- (1, 2);
	\draw (0.0,1.5) circle (0.15);
	\fill (0.5,1.0) circle (0.15);
	\draw (1.0,1.5) circle (0.15);
	\draw (1.5,1.0) circle (0.15);
	\draw (0.0,0.5) circle (0.15);
	\draw (0.5,0.0) circle (0.15);
	\draw (1.0,0.5) circle (0.15);
	\draw (1.5,0.0) circle (0.15);
\end{tikzpicture}
	& 	\begin{tikzpicture}[scale = 0.5, baseline={([yshift=-.8ex]current bounding box.center)}]
	\draw[cap = rect] (0, 4) -- (0, 0) --  (4, 0);
	\draw[thin, color=gray, dotted] (4, 0) -- (4, 4) -- (0, 4);
	\draw (0, 3) -- (4, 3);
	\draw (0, 2) -- (4, 2);
	\draw (0, 1) -- (4, 1);
	\draw (3, 0) -- (3, 4);
	\draw (2, 0) -- (2, 4);
	\draw (1, 0) -- (1, 4);
	\draw (0.0,3.5) circle (0.15);
	\fill (0.5,3.0) circle (0.15);
	\fill (1.0,3.5) circle (0.15);
	\draw (1.5,3.0) circle (0.15);
	\draw (2.0,3.5) circle (0.15);
	\fill (2.5,3.0) circle (0.15);
	\draw (3.0,3.5) circle (0.15);
	\draw (3.5,3.0) circle (0.15);
	\draw (0.0,2.5) circle (0.15);
	\fill (0.5,2.0) circle (0.15);
	\fill (1.0,2.5) circle (0.15);
	\fill (1.5,2.0) circle (0.15);
	\draw (2.0,2.5) circle (0.15);
	\draw (2.5,2.0) circle (0.15);
	\draw (3.0,2.5) circle (0.15);
	\draw (3.5,2.0) circle (0.15);
	\draw (0.0,1.5) circle (0.15);
	\draw (0.5,1.0) circle (0.15);
	\draw (1.0,1.5) circle (0.15);
	\fill (1.5,1.0) circle (0.15);
	\draw (2.0,1.5) circle (0.15);
	\draw (2.5,1.0) circle (0.15);
	\fill (3.0,1.5) circle (0.15);
	\fill (3.5,1.0) circle (0.15);
	\draw (0.0,0.5) circle (0.15);
	\draw (0.5,0.0) circle (0.15);
	\draw (1.0,0.5) circle (0.15);
	\draw (1.5,0.0) circle (0.15);
	\draw (2.0,0.5) circle (0.15);
	\draw (2.5,0.0) circle (0.15);
	\fill (3.0,0.5) circle (0.15);
	\draw (3.5,0.0) circle (0.15);
\end{tikzpicture}
	&	\begin{tikzpicture}[scale = 0.5, baseline={([yshift=-.8ex]current bounding box.center)}]
	\clip( -0.5, -0.5) rectangle ( 8.5, 8.5);
	\draw[cap = rect] (0, 8) -- (0, 0) --  (8, 0);
	\draw[thin, color=gray, dotted] (8, 0) -- (8, 8) -- (0, 8);
	\draw (0, 7) -- (8, 7);
	\draw (0, 6) -- (8, 6);
	\draw (0, 5) -- (8, 5);
	\draw (0, 4) -- (8, 4);
	\draw (0, 3) -- (8, 3);
	\draw (0, 2) -- (8, 2);
	\draw (0, 1) -- (8, 1);
	\draw (7, 0) -- (7, 8);
	\draw (6, 0) -- (6, 8);
	\draw (5, 0) -- (5, 8);
	\draw (4, 0) -- (4, 8);
	\draw (3, 0) -- (3, 8);
	\draw (2, 0) -- (2, 8);
	\draw (1, 0) -- (1, 8);
	\draw (0.0,7.5) circle (0.15);
	\fill (0.5,7.0) circle (0.15);
	\fill (1.0,7.5) circle (0.15);
	\draw (1.5,7.0) circle (0.15);
	\fill (2.0,7.5) circle (0.15);
	\fill (2.5,7.0) circle (0.15);
	\draw (3.0,7.5) circle (0.15);
	\draw (3.5,7.0) circle (0.15);
	\draw (4.0,7.5) circle (0.15);
	\fill (4.5,7.0) circle (0.15);
	\fill (5.0,7.5) circle (0.15);
	\draw (5.5,7.0) circle (0.15);
	\draw (6.0,7.5) circle (0.15);
	\fill (6.5,7.0) circle (0.15);
	\draw (7.0,7.5) circle (0.15);
	\draw (7.5,7.0) circle (0.15);
	\draw (0.0,6.5) circle (0.15);
	\fill (0.5,6.0) circle (0.15);
	\fill (1.0,6.5) circle (0.15);
	\fill (1.5,6.0) circle (0.15);
	\fill (2.0,6.5) circle (0.15);
	\draw (2.5,6.0) circle (0.15);
	\draw (3.0,6.5) circle (0.15);
	\draw (3.5,6.0) circle (0.15);
	\draw (4.0,6.5) circle (0.15);
	\fill (4.5,6.0) circle (0.15);
	\fill (5.0,6.5) circle (0.15);
	\fill (5.5,6.0) circle (0.15);
	\draw (6.0,6.5) circle (0.15);
	\draw (6.5,6.0) circle (0.15);
	\draw (7.0,6.5) circle (0.15);
	\draw (7.5,6.0) circle (0.15);
	\draw (0.0,5.5) circle (0.15);
	\draw (0.5,5.0) circle (0.15);
	\draw (1.0,5.5) circle (0.15);
	\fill (1.5,5.0) circle (0.15);
	\fill (2.0,5.5) circle (0.15);
	\draw (2.5,5.0) circle (0.15);
	\fill (3.0,5.5) circle (0.15);
	\fill (3.5,5.0) circle (0.15);
	\draw (4.0,5.5) circle (0.15);
	\draw (4.5,5.0) circle (0.15);
	\draw (5.0,5.5) circle (0.15);
	\fill (5.5,5.0) circle (0.15);
	\draw (6.0,5.5) circle (0.15);
	\draw (6.5,5.0) circle (0.15);
	\fill (7.0,5.5) circle (0.15);
	\fill (7.5,5.0) circle (0.15);
	\draw (0.0,4.5) circle (0.15);
	\fill (0.5,4.0) circle (0.15);
	\draw (1.0,4.5) circle (0.15);
	\fill (1.5,4.0) circle (0.15);
	\fill (2.0,4.5) circle (0.15);
	\fill (2.5,4.0) circle (0.15);
	\fill (3.0,4.5) circle (0.15);
	\fill (3.5,4.0) circle (0.15);
	\draw (4.0,4.5) circle (0.15);
	\draw (4.5,4.0) circle (0.15);
	\draw (5.0,4.5) circle (0.15);
	\draw (5.5,4.0) circle (0.15);
	\draw (6.0,4.5) circle (0.15);
	\draw (6.5,4.0) circle (0.15);
	\fill (7.0,4.5) circle (0.15);
	\draw (7.5,4.0) circle (0.15);
	\draw (0.0,3.5) circle (0.15);
	\fill (0.5,3.0) circle (0.15);
	\fill (1.0,3.5) circle (0.15);
	\draw (1.5,3.0) circle (0.15);
	\draw (2.0,3.5) circle (0.15);
	\fill (2.5,3.0) circle (0.15);
	\draw (3.0,3.5) circle (0.15);
	\draw (3.5,3.0) circle (0.15);
	\draw (4.0,3.5) circle (0.15);
	\fill (4.5,3.0) circle (0.15);
	\fill (5.0,3.5) circle (0.15);
	\draw (5.5,3.0) circle (0.15);
	\fill (6.0,3.5) circle (0.15);
	\fill (6.5,3.0) circle (0.15);
	\draw (7.0,3.5) circle (0.15);
	\draw (7.5,3.0) circle (0.15);
	\draw (0.0,2.5) circle (0.15);
	\draw (0.5,2.0) circle (0.15);
	\fill (1.0,2.5) circle (0.15);
	\draw (1.5,2.0) circle (0.15);
	\draw (2.0,2.5) circle (0.15);
	\fill (2.5,2.0) circle (0.15);
	\draw (3.0,2.5) circle (0.15);
	\fill (3.5,2.0) circle (0.15);
	\draw (4.0,2.5) circle (0.15);
	\draw (4.5,2.0) circle (0.15);
	\fill (5.0,2.5) circle (0.15);
	\draw (5.5,2.0) circle (0.15);
	\fill (6.0,2.5) circle (0.15);
	\fill (6.5,2.0) circle (0.15);
	\draw (7.0,2.5) circle (0.15);
	\fill (7.5,2.0) circle (0.15);
	\draw (0.0,1.5) circle (0.15);
	\draw (0.5,1.0) circle (0.15);
	\draw (1.0,1.5) circle (0.15);
	\fill (1.5,1.0) circle (0.15);
	\draw (2.0,1.5) circle (0.15);
	\draw (2.5,1.0) circle (0.15);
	\fill (3.0,1.5) circle (0.15);
	\fill (3.5,1.0) circle (0.15);
	\draw (4.0,1.5) circle (0.15);
	\draw (4.5,1.0) circle (0.15);
	\draw (5.0,1.5) circle (0.15);
	\fill (5.5,1.0) circle (0.15);
	\fill (6.0,1.5) circle (0.15);
	\draw (6.5,1.0) circle (0.15);
	\fill (7.0,1.5) circle (0.15);
	\fill (7.5,1.0) circle (0.15);
	\draw (0.0,0.5) circle (0.15);
	\draw (0.5,0.0) circle (0.15);
	\draw (1.0,0.5) circle (0.15);
	\draw (1.5,0.0) circle (0.15);
	\draw (2.0,0.5) circle (0.15);
	\draw (2.5,0.0) circle (0.15);
	\fill (3.0,0.5) circle (0.15);
	\draw (3.5,0.0) circle (0.15);
	\draw (4.0,0.5) circle (0.15);
	\draw (4.5,0.0) circle (0.15);
	\draw (5.0,0.5) circle (0.15);
	\draw (5.5,0.0) circle (0.15);
	\fill (6.0,0.5) circle (0.15);
	\draw (6.5,0.0) circle (0.15);
	\fill (7.0,0.5) circle (0.15);
	\draw (7.5,0.0) circle (0.15);
\end{tikzpicture} \\
		$S_1$
	&	$S_2$
	&	$S_3$
\end{tabular}
\caption{The first $S_i$ patterns decorated with creases.}
\label{fig:firstSndec}
\end{figure}

From Figure~\ref{fig:firstSnFold} and Figure~\ref{fig:firstSndec} we see that $S_n$ represents the upper right quadrant of $\mathcal{S}_{n+1}$.

\section{Sets of Patterns}
\label{sec:patterns}

In this section we turn our attention to defining what we mean by a pattern, and look at sets of patterns and the structure of such sets. 

Recall the definition of the $T_n$, and $S_n$ from \eqref{eq:DefTnBlock} and \eqref{eq:DefSnBlock}. We can see these structures (the $T_n$s and $S_n$s) as matrices on their respective alphabet. According to the language used in the field of tilings, we say that submatrices of the $T_n$s and $S_n$s are called \emph{patterns} or \emph{subpatterns}. That is, a pattern is rectangular and without holes. (In the literature the term \emph{patch} (see \cite{BG13,FHG}) is also commonly used for this.) Clearly, any $T_n$ or $S_n$ is also a pattern. For completeness, we also say that $T$ and $S$ are also  patterns (infinite ones). We also adopt the notations used for matrices, with rows and columns. This means we can describe a finite pattern $X$ via its entries, that is, $X_{r,c}$ is the entry in $X$ at row $r$ and column $c$.

For a pattern $X$ (finite or infinite), we let $P(X, m\times n)$, where $m$ and $n$ are positive integers, be the set of all $m\times n$ patterns that occur somewhere in $X$. In the case of $X$ being a finite pattern, we use the notation $X[r, c, m \times n]$ to denote the $m \times n$ subpattern of $X$ that has its upper left corner at row $r$ and column $c$ in $X$. The notation $|\cdot|$ denotes the cardinality of a set. Note here that the quantity 
\begin{equation}
\label{def:An}	
	A_n := |P(S, n\times n)|
\end{equation}
is what we mean by the number of patterns of size $n\times n$ in Theorem~\ref{thm:main}. With the help of the just introduced notations, we can now state a first result on the sets of patterns, implied by the block recursion given in \eqref{eq:DefTnBlock}. 

\begin{lemma}
	\label{lemma:inclusion}
	Let $n\geq 0$. Then $T_n \in P(T_{n+1}, 2^n\times 2^n)$.\hfill$\square$ 
\end{lemma}

The Lemma~\ref{lemma:inclusion} implies that the chain of nested sets of subpatterns,
\[
   P(T_0,m \times m) \subseteq \cdots  \subseteq P(T_n,m \times m) \subseteq P(T_{n+1},m \times m) \subseteq \cdots, 
\]
is monotonic including in $n$, (if $m\leq2^n$). Next, we show that this chain is strictly monotonic including until all possible subpatterns are contained. 

\begin{lemma}
	\label{lemma:uniqueplateau}
	Let $m\geq1$. If there is an $n\geq0$ such that 
	\begin{equation}
		\label{eq:plateauassumption}
		P(T_{n}, m \times m) = P(T_{n+1}, m \times m),
	\end{equation}
	with $m\leq 2^n$, then 
	\begin{equation}
		\label{eq:plateauresult}
		P(T_{n}, m \times m) = P(T_{n+k}, m \times m)
	\end{equation}
	for all integers $k \geq 1$, and in particular $P(T_{n}, m \times m) = P(T, m \times m)$. 
\end{lemma}

\begin{proof}
We give a proof by induction on $k$ in \eqref{eq:plateauresult}. The basis case, $k = 1$, is direct from the assumption \eqref{eq:plateauassumption}. Assume for induction that \eqref{eq:plateauresult} holds for $1\leq k \leq p$. 
	
For the induction step, $k = p+1$, consider a pattern $a \in P(T_{n+p+1}, m \times m)$. Then there is a pattern $b \in P(T_{n+p}, m\times m)$ such that $a$ is a subpattern of $\mu(b)$. By the induction assumption we have that $b \in P(T_{n+p-1}, m \times m)$. This implies
\[  
	a \in P(\mu(b),m\times m) \subseteq P(T_{n+p}, m\times m).
\]
Therefore $P(T_{n+p}, m \times m) \supseteq P(T_{n+p+1}, m \times m)$, and by Lemma~\ref{lemma:inclusion} it follows that 
\[
	P(T_{n+p}, m\times m) = P(T_{n+p+1}, m\times m),
\]
which completes the induction. 
\end{proof}

\begin{example}
\label{ex:patterncount}
By an enumeration we find
\[
	P(T_5, 2 \times 2) = P(T_6, 2\times 2),
\]
with $|P(T_5, 2\times 2)| = 76$. Lemma~\ref{lemma:uniqueplateau} now implies that $P(T_5, 2\times 2) = P(T, 2\times 2)$, so we can find all $2\times 2$ patterns in $T$ by just looking at patterns in $T_5$. In the same way, continuing the enumeration and applying Lemma~\ref{lemma:uniqueplateau}, we find 
\[  
	P(T_6, 4\times 4) = P(T_7, 4\times4) =  P(T, 4\times4),
\]
with $|P(T, 4\times 4)| = 316$. As a consequence, we clearly also have $P(T_6, 3\times3) = P(T, 3\times3)$ without any further enumerations. This because $T_6$ contains all $4\times4$ patterns, and therefore it must also contain all $3\times3$ patterns. \hfill$\diamond$
\end{example}

The elements of the set $P(T,m \times n)$ can be split into sets depending on their element's position relative to the underlying structure of supertiles of size $2\times 2$. For $i,j\in\{1,2\}$ we define the sets  
\begin{equation}
\label{eq:defofPijT}
	P_{i,j}(T, m \times n) 
	:= \big\{ \mu(x)[i,j,m \times n] : x \in P(T, m \times n)\big\}. 
\end{equation}
The definition in \eqref{eq:defofPijT} can be extend to all positive indices via
\[
	P_{i+2s,j+2t}(T, m \times n) := P_{i,j}(T, m \times n),
\]
where $s,t \in \mathbb{N}$. Note that $P_{i,j}$ in \eqref{eq:defofPijT} is defined for patterns on the alphabet $\mathcal{A}$. Analogous to the definition in \eqref{eq:defofPijT} we define 
\[
	P_{i,j}(S, m \times n) 
	:= \big\{ \phi(x)[i,j,m \times n] : x \in P(T, m \times n)\big\} 
\]
for patterns on the alphabet $\mathcal{B}$. Here we can also extend the definition to any positive indices $s,t$ via
\[
	P_{i+2s,j+2t}(T, m \times n) := P_{i,j}(T, m \times n).
\]
From the definition in \eqref{eq:defofPijT} it is clear that 
\[
	P(T, m \times n) = \bigcup_{i,j\in\{1,2\}} P_{i,j}(T, m \times n),
\]
as any $x\in P(T, m \times n)$  must be in at least one $P_{i,j}(T, m \times n)$. Moreover, it is by construction also clear that each of the sets $P_{i,j}(T, m \times n)$ are non-empty.  
 
\begin{lemma}
\label{lemma:disjointPijT}
Let $m,n\geq1$. Then 
\begin{equation}
\label{eq:disjointPijT}
	P(T, m \times n) = \bigcup_{i,j\in\{1,2\}} P_{i,j}(T, m \times n), 
\end{equation}
and the sets in the union on the right hand side of \eqref{eq:disjointPijT} are non-empty and pairwise disjoint.
\end{lemma}

\begin{proof}
The only non-clear part is the disjointedness. From the definition of $\mu$ in \eqref{def:mu}, we see that any 
letter $a\in \mathcal{A}$ occurs at one and only one position in the mapped blocks. That is, for $x\in\mathcal{A}$,
\[
	\mu(x)_{i,j} \in 
	\begin{cases}
	\{\mathtt{A}, \mathtt{B}, \mathtt{I}, \mathtt{J}\}, & \textnormal{if } (i,j) = (1,1), \\
	\{\mathtt{E}, \mathtt{F}, \mathtt{M}, \mathtt{N}\}, & \textnormal{if } (i,j) = (1,2), \\
	\{\mathtt{G}, \mathtt{H}, \mathtt{O}, \mathtt{P}\}, & \textnormal{if } (i,j) = (2,1), \\
	\{\mathtt{C}, \mathtt{D}, \mathtt{K}, \mathtt{L}\}, & \textnormal{if } (i,j) = (2,2).
	\end{cases}
\]
This implies that the $P_{i,j}(T, m\times n)$ are disjoint for $i,j \in \{1,2\}$. 
\end{proof}

\begin{example}
\label{ex:disjointPijS}
In Example~\ref{ex:patterncount} we saw that all patterns of size $4\times4 $ are found in $T_6$. This leads to 
that we can find the sets $P(S, 3 \times 3)$ via a finite enumeration. It is now straight forward to find the sets $P_{i,j}(S, 3 \times 3)$. We obtain for $n=3$,
\begin{equation}	
\label{eq:disjointPijS}
	P(S, n \times n) = \bigcup_{i,j\in\{1,2\}} P_{i,j}(S, n \times n), 
\end{equation}
where the sets in the union on the right hand side of \eqref{eq:disjointPijS} are non-empty and pairwise disjoint. In contrast to sets of patterns on $\mathcal{A}$, see Lemma~\ref{lemma:disjointPijT}, the equality in \eqref{eq:disjointPijS} does not hold for $n=1,2$. 
\hfill$\diamond$
\end{example}

\begin{lemma}
\label{lemma:disjointPijS}
Let $m,n\geq3$. Then 
\begin{equation}
	\label{eq:lemma:disjointPijS}
	P(S, m \times n) = \bigcup_{i,j\in\{1,2\}} P_{i,j}(S, m \times n), 
\end{equation}
and the sets in the union on the right hand side of \eqref{eq:lemma:disjointPijS} are non-empty and pairwise disjoint.
\end{lemma}

\begin{proof}
Assume for contradiction that there are $m,n\geq3$ and two distinct pair of indices $(i_1, j_1), (i_2,j_2) \in\{1,2\}^2$ such that there is a pattern 
\[
	x \in P_{i_1,j_1}(S, m \times n) \, \bigcup \, P_{i_2,j_2}(S, m \times n).
\]
Then the pattern $ x' = x[1,1, 3\times 3]$ must be in the intersection 
\[	P_{i_1,j_1}(S, 3 \times 3) \, \bigcup \, P_{i_2,j_2}(S, 3 \times 3),
\]
which contradicts what we saw in Example~\ref{ex:disjointPijS}.
\end{proof}

By looking at the definition of the substitutions $\mu$ and $\phi$ from \eqref{def:mu} and \eqref{def:phi} one sees that for any $a\in \mathcal{A}$ we have correspondences;
\[
\begin{array}{rcl}
	\begin{array}{rcl}
		\mu(a)_{1,1} = \mathtt{A} &\Leftrightarrow &\phi(a)_{1,1} = \mathtt{0}, \\
		\mu(a)_{1,1} = \mathtt{B} &\Leftrightarrow &\phi(a)_{1,1} = \mathtt{1}, \\
		\mu(a)_{1,1} = \mathtt{I} &\Leftrightarrow &\phi(a)_{1,1} = \mathtt{2},	\\
		\mu(a)_{1,1} = \mathtt{J} &\Leftrightarrow &\phi(a)_{1,1} = \mathtt{3},	\\
	\end{array}
	&&
	\begin{array}{rcl}
		\mu(a)_{1,2} = \mathtt{E} &\Leftrightarrow &\phi(a)_{1,2} = \mathtt{0}, \\
		\mu(a)_{1,2} = \mathtt{F} &\Leftrightarrow &\phi(a)_{1,2} = \mathtt{1}, \\
		\mu(a)_{1,2} = \mathtt{M} &\Leftrightarrow &\phi(a)_{1,2} = \mathtt{2},	\\
		\mu(a)_{1,2} = \mathtt{N} &\Leftrightarrow &\phi(a)_{1,2} = \mathtt{3},	\\
	\end{array}
\\
\\
	\begin{array}{rcl}
		\mu(a)_{2,1} = \mathtt{G} &\Leftrightarrow &\phi(a)_{2,1} = \mathtt{0}, \\
		\mu(a)_{2,1} = \mathtt{H} &\Leftrightarrow &\phi(a)_{2,1} = \mathtt{1}, \\
		\mu(a)_{2,1} = \mathtt{O} &\Leftrightarrow &\phi(a)_{2,1} = \mathtt{2},	\\
		\mu(a)_{2,1} = \mathtt{P} &\Leftrightarrow &\phi(a)_{2,1} = \mathtt{3},	\\
	\end{array}
	&&
	\begin{array}{rcl}
		\mu(a)_{2,2} = \mathtt{C} &\Leftrightarrow &\phi(a)_{2,2} = \mathtt{0}, \\
		\mu(a)_{2,2} = \mathtt{D} &\Leftrightarrow &\phi(a)_{2,2} = \mathtt{1}, \\
		\mu(a)_{2,2} = \mathtt{K} &\Leftrightarrow &\phi(a)_{2,2} = \mathtt{2},	\\
		\mu(a)_{2,2} = \mathtt{L} &\Leftrightarrow &\phi(a)_{2,2} = \mathtt{3}.	
	\end{array}
	\end{array}
\]
This implies the following lemma. 

\begin{lemma}
\label{lemma:muequalsphi}
Let $m,n\geq 1$, and $x,y\in P(T, m\times n)$. Then 
\[	\mu(x)[i,j, m\times n] = \mu(y)[i,j, m\times n]
\]
if and only if 
\[	\phi(x)[i,j, m\times n] = \phi(y)[i,j, m\times n],
\]
where $i,j\in{1,2}$. 
\end{lemma}

We can now state and conclude a main results of this section. 

\begin{theorem}
\label{thm:PTequalsPS}
Let $n\geq 3$. Then $|P(T, n\times n)| = |P(S, n\times n)|$.
\end{theorem}

\begin{proof}
Combine Lemma~\ref{lemma:disjointPijT}, Lemma~\ref{lemma:disjointPijS}, and Lemma~\ref{lemma:muequalsphi}.
\end{proof}

The implication of the above Theorem~\ref{thm:PTequalsPS} is that in order to find the number of square patterns in $S$, as stated in the main Theorem~\ref{thm:main}, (see also \eqref{def:An}), we can look at patterns in $T$, if $n\geq3$.  

The sets $P_{i,j}$ can be split into sets depending on their element's position in the structure of supertiles. Similar to \eqref{eq:defofPijT} we define, for $i,j\in\{1,2,3,4\}$, the sets 
\begin{equation}
\label{eq:defofQijT}
	Q_{i,j}(T, m \times n) 
	:= \big\{ \mu^2(x)[i,j,m \times n] : x \in P(T, m \times n)\big\}. 
\end{equation}
The definition in \eqref{eq:defofQijT} can be extend to all positive indices via
\[
	Q_{i+4s,j+4t}(T, m \times n) := Q_{i,j}(T, m \times n),
\]
where $s,t \in \mathbb{N}$. By definition it is clear that 
\[
	P_{i,j}(T, m \times n) = \bigcup_{k,l\in\{0,2\}} Q_{i+k,j+l}(T, m \times n)
\]
where the sets on the right hand side are non-empty. 

\begin{lemma}
\label{lemma:disjointQij}
Let $m,n \geq 2$ and $i,j\in\{1,2\}$. Then 
\begin{equation}
\label{eq:disjointQij}
	P_{i,j}(T, m \times n) = \bigcup_{k,l\in\{0,2\}} Q_{i+k,j+l}(T, m \times n), 
\end{equation}
and the sets in the union on the right hand side of \eqref{eq:disjointQij} are non-empty and pairwise disjoint.
\end{lemma}

\begin{proof}
A straight forward enumeration, similar to the one discussed in Example~\ref{ex:disjointPijS}, shows that \eqref{eq:disjointQij} holds for $m=n=2$, and the involved $Q$-sets on the right hand side are disjoint. Assume for contradiction that there are integers $m,n\geq 2$, an index  $(i,j)\in\{1,2\}^2$, and two distinct pair of indices $(k_1, l_1), (k_2,l_2) \in\{0,2\}^2$ such that there is a pattern 
\[
	x \in Q_{i+k_1,j+l_1}(T, m \times n) \, \bigcup \, Q_{i+k_2,j+l_2}(T, m \times n).
\]
Then the pattern $ x' = x[1,1, 2\times 2]$ must be in the intersection 
\[	Q_{i+k_1,j+l_1}(T, 2 \times 2) \, \bigcup \, Q_{i+k_2,j+l_2}(T, 2 \times 2),
\]
which contradicts the case seen in the enumeration mentioned in at the beginning of the proof.
\end{proof}

\section{Extensions}
\label{sec:extensions}

The aim of this section is to discuss bijective connections between sets of patterns. We start by introducing the following shorthand notation 
\begin{equation}
\label{eq:defabc}
\renewcommand{\arraystretch}{1.3}%
	\begin{array}{r@{\ }l}	 
		a_{i,j}(n) &:= |P_{i,j}(T,n \times n)|, \\
		b_{i,j}(n) &:= |P_{i,j}(T,n \times (n+1))|, \\
		c_{i,j}(n) &:= |P_{i,j}(T,(n+1) \times n)|.
	\end{array}
\end{equation}
 
\begin{example}
\label{ex:extendsion}
An enumeration shows that we have 
\[	|P_{2,1}(T,3\times3)| = |P_{1,1}(T,4\times4)|.
\]
That is, for any element $u \in P_{2,1}(T,3\times3)$ there is a unique $v\in P_{1,1}(T,4\times4)$ with $u = v[2,1,3\times3]$. See Figure~\ref{fig:extensionP} for an illustration of this extension of the pattern $u$ to $v$.  \hfill$\diamond$   
\end{example}

\begin{figure}[ht]
\centering
	\begin{tikzpicture}[scale = 0.3]
		\draw[color=black, fill=gray, fill opacity=0.3] (2, 2) rectangle ++ (4,4) ;
		\draw[color=black, fill=cyan, fill opacity=0.6] (2, 2) rectangle ++ (3,3);
		\draw[step=1, black, dotted]    (-0.5,-0.5) grid (8.5,8.5);
		\draw[step=2, black, cap=rect]  (-0.5,-0.5) grid (8.5,8.5);
	\end{tikzpicture}
	\caption{An element $u \in P_{2,1}(T, 3\times 3)$, the blue rectangle, can be uniquely extended to an element $v\in P_{1,1}(T, 4\times 4)$, the gray rectangle (here partly covered by the blue). }
    \label{fig:extensionP}
\end{figure}
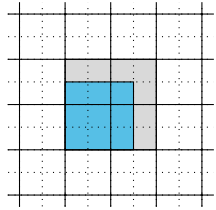

In the following lemma, we summarize the consequences of the enumeration and observation we made in Example~\ref{ex:extendsion}.

\begin{lemma}
\label{lemma:first_abc_equalties}
Let $n\geq2$. Then 
\[
\renewcommand{\arraystretch}{1.3}
	\begin{array}{rr@{\ }l}
		\mathit{1}. & a_{1,2}(2n)   &= b_{1,2}(2n) = b_{2,2}(2n-1),\\ %:752
		\mathit{2}. & a_{2,1}(2n)   &= c_{1,1}(2n) = c_{2,1}(2n-1), \\ %:752
		\mathit{3}. & a_{2,2}(2n)   &= a_{1,2}(2n+1) = c_{1,2}(2n) = b_{2,2}(2n), \\ %:752 
		\mathit{4}. & a_{1,1}(2n+1) &= b_{1,1}(2n+1) = b_{2,1}(2n), \\ %:888
		\mathit{5}. & a_{2,1}(2n+1) &= a_{1,1}(2n+2) = b_{2,1}(2n+1) = c_{1,1}(2n+1),\\ %:1000
		\mathit{6}. & a_{2,2}(2n+1) &= c_{1,2}(2n+1) = c_{2,2}(2n). %:856
	\end{array}
\]
\end{lemma}

\begin{proof}
Let us consider the first equality in the statement of this lemma, that is, $a_{1,2}(2n)= b_{1,2}(2n)$. Assume for contradiction that there is an $n$ such that this equality doesn't hold. Then there is a pair of distinct patterns $u, v\in P_{1,2}(T, n\times (n+1))$ such that $u[1,1,n\times n] = v[1,1,n\times n]$. Then we can find $s,t \in \mathbb{N}$ such that 
\[
	u[1+2s,2+2t, 3\times 4] \neq v[1+2s,1+2t,3\times 4]
\]
but with
\[
	u[1+2s,1+2t, 3\times 3] = v[1+2s,1+2t,3\times 3].
\]
That is, we have a contradiction to what we saw in Example~\ref{ex:extendsion}. The remaining equalities follow in the same way. See Figures~\ref{fig:extensionP1} through Figure~\ref{fig:extensionP6}.
\end{proof}

\begin{figure}[ht]
\centering
	\begin{tikzpicture}[scale = 0.275]
		\begin{scope}
			\draw[color=black, fill=gray, fill opacity=0.3] (1, 2) rectangle ++ (5,4);
			\draw[color=black, fill=cyan, fill opacity=0.6] (1, 2) rectangle ++ (4,4);
			\draw[step=1, black, dotted]    (-0.5,-0.5) grid (8.5,8.5);
			\draw[step=2, black, cap=rect]  (-0.5,-0.5) grid (8.5,8.5);
			\draw(4, -2) node{\footnotesize $a_{1,2}(2n) = b_{1,2}(2n)$};
		\end{scope}
		\begin{scope}[xshift = 15cm]
			\draw[color=black, fill=gray, fill opacity=0.3] (1, 2) rectangle ++ (4,4);
			\draw[color=black, fill=cyan, fill opacity=0.6] (1, 2) rectangle ++ (4,3);
			\draw[step=1, black, dotted]    (-0.5,-0.5) grid (8.5,8.5);
			\draw[step=2, black, cap=rect]  (-0.5,-0.5) grid (8.5,8.5);
			\draw(4, -2) node{\footnotesize $a_{1,2}(2n) = b_{2,2}(2n-1) $};
		\end{scope}
	\end{tikzpicture}
	\caption{On the left, an element $u \in P_{1,2}(T, 4\times 4)$, the blue rectangle, can be uniquely extended to an element $u'\in P_{1,2}(T, 4\times 5)$, the gray rectangle. On the right, an element $v \in P_{2,2}(T, 3\times 4)$, the blue rectangle, can be uniquely extended to an element $v'\in P_{1,2}(T, 4\times 4)$, the gray rectangle.}
    \label{fig:extensionP1}
\end{figure}
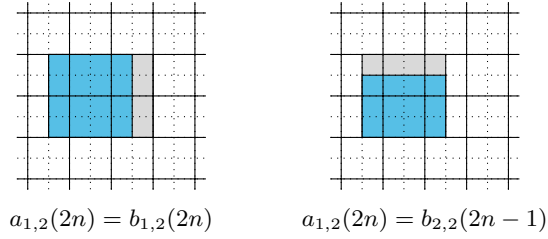

\begin{figure}[ht]
\centering
	\begin{tikzpicture}[scale = 0.275]
		\begin{scope}[xshift = 0cm]
			\draw[color=black, fill=gray, fill opacity=0.3] (2, 1) rectangle ++ (4,5);
			\draw[color=black, fill=cyan, fill opacity=0.6] (2, 1) rectangle ++ (4,4);
			\draw[step=1, black, dotted]    (-0.5,-0.5) grid (8.5,8.5);
			\draw[step=2, black, cap=rect]  (-0.5,-0.5) grid (8.5,8.5);
			\draw(4, -2) node{\footnotesize $ a_{2,1}(2n) = c_{1,1}(2n) $};
		\end{scope}
		\begin{scope}[xshift = 15cm]
			\draw[color=black, fill=gray, fill opacity=0.3] (2, 1) rectangle ++ (4,4);
			\draw[color=black, fill=cyan, fill opacity=0.6] (2, 1) rectangle ++ (3,4);
			\draw[step=1, black, dotted]    (-0.5,-0.5) grid (8.5,8.5);
			\draw[step=2, black, cap=rect]  (-0.5,-0.5) grid (8.5,8.5);
			\draw(4, -2) node{\footnotesize $a_{2,1}(2n) = c_{2,1}(2n-1)$};
		\end{scope}
	\end{tikzpicture}
	\caption{On the left, an element $u \in P_{2,1}(T, 4\times 4)$, the blue rectangle, can be uniquely extended to an element $u'\in P_{1,1}(T, 5\times 4)$, the gray rectangle. On the right, an element $v \in P_{2,1}(T, 4\times 3)$, the blue rectangle, can be uniquely extended to an element $v'\in P_{2,1}(T, 4\times 4)$, the gray rectangle. }
    \label{fig:extensionP2}
\end{figure}
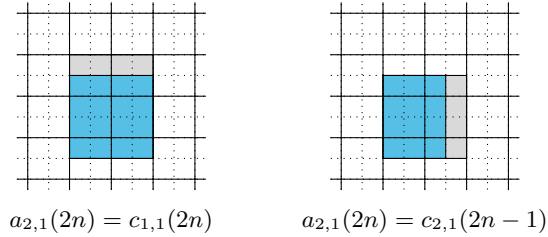

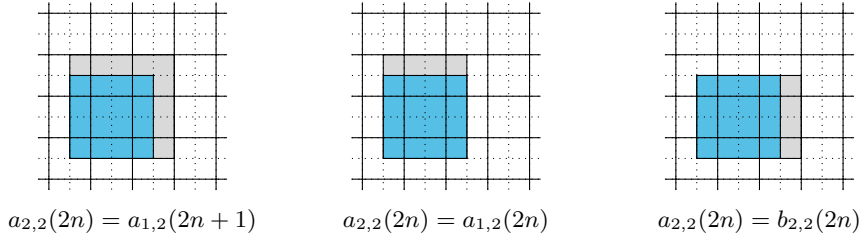
\begin{figure}[ht]
\centering
	\begin{tikzpicture}[scale = 0.275]
		\begin{scope}
			\draw[color=black, fill=gray, fill opacity=0.3] (1, 1) rectangle ++ (5,5);
			\draw[color=black, fill=cyan, fill opacity=0.6] (1, 1) rectangle ++ (4,4);
			\draw[step=1, black, dotted]    (-0.5,-0.5) grid (8.5,8.5);
			\draw[step=2, black, cap=rect]  (-0.5,-0.5) grid (8.5,8.5);
			\draw(4, -2) node{\footnotesize $a_{2,2}(2n) = a_{1,2}(2n+1)$};
		\end{scope}
		\begin{scope}[xshift = 15cm]
			\draw[color=black, fill=gray, fill opacity=0.3] (1, 1) rectangle ++ (4,5);
			\draw[color=black, fill=cyan, fill opacity=0.6] (1, 1) rectangle ++ (4,4);
			\draw[step=1, black, dotted]    (-0.5,-0.5) grid (8.5,8.5);
			\draw[step=2, black, cap=rect]  (-0.5,-0.5) grid (8.5,8.5);
			\draw(4, -2) node{\footnotesize $a_{2,2}(2n) = a_{1,2}(2n)$};
		\end{scope}
		\begin{scope}[xshift = 30cm]
			\draw[color=black, fill=gray, fill opacity=0.3] (1, 1) rectangle ++ (5,4);
			\draw[color=black, fill=cyan, fill opacity=0.6] (1, 1) rectangle ++ (4,4);
			\draw[step=1, black, dotted]    (-0.5,-0.5) grid (8.5,8.5);
			\draw[step=2, black, cap=rect]  (-0.5,-0.5) grid (8.5,8.5);
			\draw(4, -2) node{\footnotesize $a_{2,2}(2n) = b_{2,2}(2n)$};
		\end{scope}
	\end{tikzpicture}
	\caption{On the left, an element $u \in P_{2,2}(T, 4\times 4)$, the blue rectangle, can be uniquely extended to an element $u'\in P_{1,2}(T, 5\times 5)$, the gray rectangle. In the center, an element $w \in P_{2,2}(T, 4\times 4)$, the blue rectangle, can be uniquely extended to an element $w'\in P_{1,2}(T, 5\times 4)$, the gray rectangle. On the right, an element $v \in P_{2,2}(T, 4\times 4)$, the blue rectangle, can be uniquely extended to an element $v'\in P_{2,2}(T, 4\times 5)$, the gray rectangle. }
    \label{fig:extensionP3}
\end{figure}

\begin{figure}[ht]
\centering
	\begin{tikzpicture}[scale = 0.275]
		\begin{scope}
			\draw[color=black, fill=gray, fill opacity=0.3] (2, 1) rectangle ++ (6,5);
			\draw[color=black, fill=cyan, fill opacity=0.6] (2, 1) rectangle ++ (5,5);
			\draw[step=1, black, dotted]    (-0.5,-0.5) grid (8.5,8.5);
			\draw[step=2, black, cap=rect]  (-0.5,-0.5) grid (8.5,8.5);
			\draw(4, -2) node{\footnotesize $a_{1,1}(2n+1) = b_{1,1}(2n+1)$};
		\end{scope}
		\begin{scope}[xshift = 15cm]
			\draw[color=black, fill=gray, fill opacity=0.3] (2, 1) rectangle ++ (5,5);
			\draw[color=black, fill=cyan, fill opacity=0.6] (2, 1) rectangle ++ (5,4);
			\draw[step=1, black, dotted]    (-0.5,-0.5) grid (8.5,8.5);
			\draw[step=2, black, cap=rect]  (-0.5,-0.5) grid (8.5,8.5);
			\draw(4, -2) node{\footnotesize $a_{1,1}(2n+1) = b_{2,1}(2n)$};
		\end{scope}
	\end{tikzpicture}
	\caption{On the left, an element $u \in P_{1,1}(T, 5\times 5)$, the blue rectangle, can be uniquely extended to an element $u'\in P_{1,1}(T, 5\times 6)$, the gray rectangle. On the right, an element $v \in P_{2,1}(T, 4\times 5)$, the blue rectangle, can be uniquely extended to an element $v'\in P_{1,1}(T, 5\times 5)$, the gray rectangle. }
    \label{fig:extensionP4}
\end{figure}
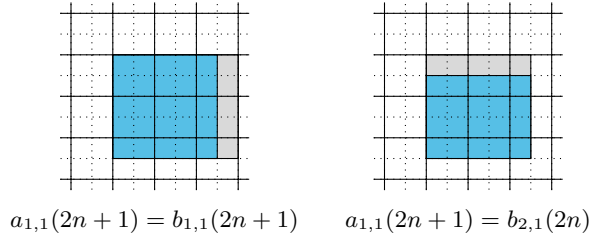

\begin{figure}[ht]
\centering
	\begin{tikzpicture}[scale = 0.275]
		\begin{scope}[xshift = 0cm]
			\draw[color=black, fill=gray, fill opacity=0.3] (2, 2) rectangle ++ (6,6);
			\draw[color=black, fill=cyan, fill opacity=0.6] (2, 2) rectangle ++ (5,5);
			\draw[step=1, black, dotted]    (-0.5,-0.5) grid (8.5,8.5);
			\draw[step=2, black, cap=rect]  (-0.5,-0.5) grid (8.5,8.5);
			\draw(4, -2) node{\footnotesize $a_{2,1}(2n+1) = a_{1,1}(2n+2)$};
		\end{scope}
		\begin{scope}[xshift = 15cm]
			\draw[color=black, fill=gray, fill opacity=0.3] (2, 2) rectangle ++ (5,6);
			\draw[color=black, fill=cyan, fill opacity=0.6] (2, 2) rectangle ++ (5,5);
			\draw[step=1, black, dotted]    (-0.5,-0.5) grid (8.5,8.5);
			\draw[step=2, black, cap=rect]  (-0.5,-0.5) grid (8.5,8.5);
			\draw(4, -2) node{\footnotesize $a_{2,1}(2n+1) = c_{1,1}(2n+1)$};
		\end{scope}
		\begin{scope}[xshift = 30cm]
			\draw[color=black, fill=gray, fill opacity=0.3] (2, 2) rectangle ++ (6,5);
			\draw[color=black, fill=cyan, fill opacity=0.6] (2, 2) rectangle ++ (5,5);
			\draw[step=1, black, dotted]    (-0.5,-0.5) grid (8.5,8.5);
			\draw[step=2, black, cap=rect]  (-0.5,-0.5) grid (8.5,8.5);
			\draw(4, -2) node{\footnotesize $a_{2,1}(2n+1) = b_{2,1}(2n+1)$};
		\end{scope}
	\end{tikzpicture}
	\caption{On the left, an element $u \in P_{2,1}(T, 5\times 5)$, the blue rectangle, can be uniquely extended to an element $u'\in P_{1,1}(T, 6\times 6)$, the gray rectangle. In the center, an element $w \in P_{2,1}(T, 5\times 4)$, the blue rectangle, can be uniquely extended to an element $w'\in P_{1,1}(T, 6\times 5)$, the gray rectangle. On the right, an element $v \in P_{2,1}(T, 5\times 5)$, the blue rectangle, can be uniquely extended to an element $v'\in P_{2,1}(T, 5\times 6)$, the gray rectangle. }
    \label{fig:extensionP5}
\end{figure}
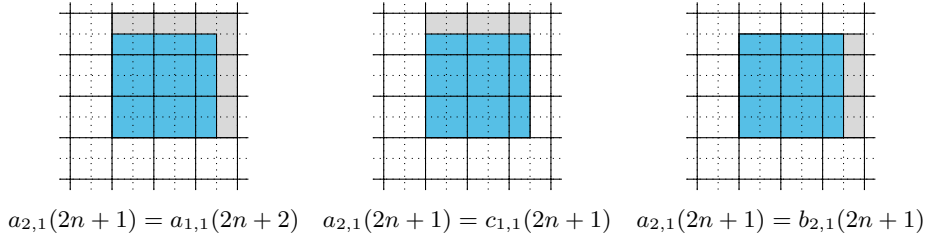

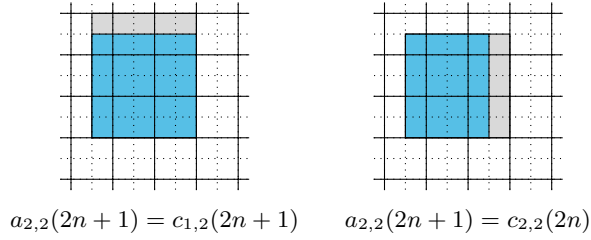
\begin{figure}[ht]
\centering
	\begin{tikzpicture}[scale = 0.275]
		\begin{scope}
			\draw[color=black, fill=gray, fill opacity=0.3] (1, 2) rectangle ++ (5,6);
			\draw[color=black, fill=cyan, fill opacity=0.6] (1, 2) rectangle ++ (5,5);
			\draw[step=1, black, dotted]    (-0.5,-0.5) grid (8.5,8.5);
			\draw[step=2, black, cap=rect]  (-0.5,-0.5) grid (8.5,8.5);
			\draw(4, -2) node{\footnotesize $a_{2,2}(2n+1) = c_{1,2}(2n+1)$};
		\end{scope}
		\begin{scope}[xshift = 15cm]
			\draw[color=black, fill=gray, fill opacity=0.3] (1, 2) rectangle ++ (5,5);
			\draw[color=black, fill=cyan, fill opacity=0.6] (1, 2) rectangle ++ (4,5);
			\draw[step=1, black, dotted]    (-0.5,-0.5) grid (8.5,8.5);
			\draw[step=2, black, cap=rect]  (-0.5,-0.5) grid (8.5,8.5);
			\draw(4, -2) node{\footnotesize $a_{2,2}(2n+1) = c_{2,2}(2n)$};
		\end{scope}
	\end{tikzpicture}
	\caption{On the left, an element $u \in P_{2,2}(T, 5\times 5)$, the blue rectangle, can be uniquely extended to an element $u'\in P_{1,2}(T, 6\times 5)$, the gray rectangle. On the right, an element $v \in P_{2,2}(T, 5\times 4)$, the blue rectangle, can be uniquely extended to an element $v'\in P_{2,2}(T, 5\times 5)$, the gray rectangle. }
    \label{fig:extensionP6}
\end{figure}

From Lemma~\ref{lemma:first_abc_equalties} we may derive the following corollary. 

\begin{corollary}
\label{cor:abcEqualities}
Let $n\geq 1$. Then 
\[
\renewcommand{\arraystretch}{1.3}
	\begin{array}[b]{rr@{\ }l}
		\mathit{1}.  & a_{1,2}(4n)   &= b_{1,2}(4n), \\
		\mathit{2}.  & a_{2,1}(4n)   &= c_{1,1}(4n), \\
		\mathit{3}.  & a_{2,2}(4n)   &= a_{1,2}(4n+1) = c_{1,2}(4n) = b_{2,2}(4n), \\
		\mathit{4}.  & a_{1,1}(4n+1) &= b_{1,1}(4n+1) = b_{2,1}(4n), \\
		\mathit{5}.  & a_{2,1}(4n+1) &= a_{1,1}(4n+2) = c_{1,1}(4n+1) = b_{2,1}(4n+1), \\
		\mathit{6}.  & a_{2,2}(4n+1) &= c_{1,2}(4n+1) = c_{2,2}(4n), \\
		\mathit{7}.  & a_{1,2}(4n+2) &= b_{1,2}(4n+2) = b_{2,2}(4n+1), \\
		\mathit{8}.  & a_{2,1}(4n+2) &= c_{1,1}(4n+2) = c_{2,1}(4n+1), \\
		\mathit{9}.  & a_{2,2}(4n+2) &= a_{1,2}(4n+3) = c_{1,2}(4n+2) = b_{2,2}(4n+2), \\
		\mathit{10}. & a_{1,1}(4n+3) &= b_{1,1}(4n+3) = b_{2,1}(4n+2), \\
		\mathit{11}. & a_{2,1}(4n+3) &= c_{1,1}(4n+3) = b_{2,1}(4n+3), \\
		\mathit{12}. & a_{2,2}(4n+3) &= c_{1,2}(4n+3) = c_{2,2}(4n+2).
	\end{array}\hfill\square
\]
\end{corollary}

The implication of Corollary~\ref{cor:abcEqualities}, is that we can reduce the number of cases of the $a,b$ and $c$-terms for which we have to find recursion relations. We will consider these recursion relations in section~\ref{sec:recursions}.

Similar to the extension we of $P$-sets we have discussed above, we can find an analogues property for the $Q$-sets, see \eqref{eq:defofQijT}. 

\begin{example}
\label{ex:extendsionQ}
An enumeration shows that we have 
\[	|Q_{4,4}(T,5\times6)| = |Q_{1,2}(T,11\times11)|.
\]
That is, for any element $u \in Q_{4,4}(T,5\times6)$ there is a unique pattern $v\in Q_{1,2}(T,11\times11)$ with $u = v[4,3,5\times6]$. See Figure~\ref{fig:extensionQ} for this extension of a pattern.
\hfill$\diamond$   
\end{example}

\begin{lemma}
\label{lemma:Qextensions}
Let $n\geq1$. Then 

\vspace{\abovedisplayskip}
\noindent
\begin{tblr}{X[c]c@{}}
$|Q_{4,4}(T,(4n+1)\times(4n+2))| = |Q_{1,2}(T,(4n+7)\times(4n+7))|.$
&
$\square$
\end{tblr}

\end{lemma}

\begin{figure}[ht]
\centering
	\begin{tikzpicture}[scale = 0.3]
		\draw[color=black, fill=gray, fill opacity=0.3] (1, 1) rectangle ++ (11, 11);
		\draw[color=black, fill=cyan, fill opacity=0.6] (3, 4) rectangle ++ (6, 5); 
		\draw[step=1, black, dotted]   (-0.5,-0.5) grid (12.5, 12.5);
		\draw[step=2, black, dashed]   (-0.5,-0.5) grid (12.5, 12.5);
		\draw[step=4, black, cap=rect] (-0.5,-0.5) grid (12.5, 12.5);
	\end{tikzpicture}
	\caption{An element $u \in Q_{4,4}(T, 5\times 6)$, the blue rectangle, can be uniquely extended to an element $v \in Q_{1,2}(T, 11\times 11)$, the gray rectangle.}
     \label{fig:extensionQ}
\end{figure}
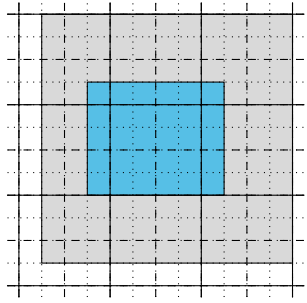

\section{Recursions}
\label{sec:recursions}

In the previous section we saw how to modify the patterns in a set without changing the cardinality of the corresponding set. In this section we apply this property to derive recursions for the quantities $a, b$ and $c$, as defined in \eqref{eq:defabc}. The equalties in Corollary~\ref{cor:abcEqualities} helps us to reduce the number of cases we have to look at here.  

Assume that $n\geq 1$. Then we have the following list of deductions, with the help of Lemma~\ref{lemma:disjointQij} and Lemma~\ref{lemma:Qextensions}. 

\begin{equation}
\label{eq:reca110}%
\text{\footnotesize$
\renewcommand{\arraystretch}{1.3}

	\caption{The modification of patterns in the sets $Q_{1+k, 1+l}(T, (4n)\times (4n))$ leading to a recursion for $a_{1,1}(4n)$. The blue rectangle indicates the starting pattern, and the gray the concluding one, compare \eqref{eq:reca110}.}
	\label{fig:reca110}
\end{figure}

\begin{equation}
\label{eq:reca120}%
\text{\footnotesize$
\renewcommand{\arraystretch}{1.3}
%
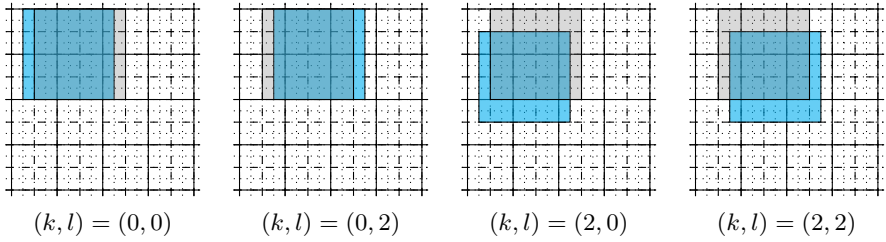
	\caption{The modification of patterns in the sets $Q_{1+k, 2+l}(T, (4n)\times (4n))$ leading to a recursion for $a_{1,2}(4n)$. The blue rectangle indicates the starting pattern, and the gray the concluding one, compare \eqref{eq:reca120}.}
	\label{fig:reca120}
\end{figure}

\begin{equation}
\label{eq:reca210}%
\text{\footnotesize$
\renewcommand{\arraystretch}{1.3}
%
	\caption{The modification of patterns in the sets $Q_{2+k, 1+l}(T, (4n)\times (4n))$ leading to a recursion for $a_{2,1}(4n)$. The blue rectangle indicates the starting pattern, and the gray the concluding one, compare \eqref{eq:reca210}.}
	\label{fig:reca210}
\end{figure}

\begin{equation}
\label{eq:reca220}%
\text{\footnotesize$
\renewcommand{\arraystretch}{1.3}
%
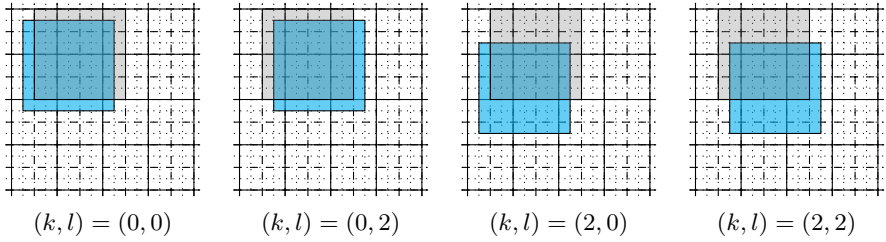
	\caption{The modification of patterns in the sets $Q_{2+k, 2+l}(T, (4n)\times (4n))$ leading to a recursion for $a_{2,2}(4n)$. The blue rectangle indicates the starting pattern, and the gray the concluding one, compare \eqref{eq:reca220}.}
	\label{fig:reca220}
\end{figure}

\begin{equation}
\label{eq:recb110}%
\text{\footnotesize$
\renewcommand{\arraystretch}{1.3}
%
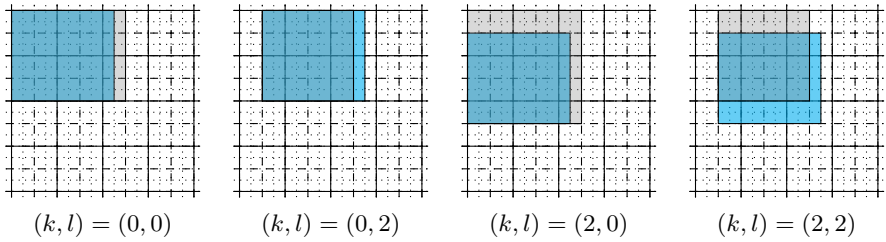
	\caption{The modification of patterns in the sets $Q_{1+k, 1+l}(T, (4n)\times (4n+1))$ leading to a recursion for $b_{1,1}(4n)$. The blue rectangle indicates the starting pattern, and the gray the concluding one, compare \eqref{eq:recb110}.}
	\label{fig:recb110}
\end{figure}

\begin{equation}
\label{eq:recc210}%
\text{\footnotesize$
\renewcommand{\arraystretch}{1.3}
%
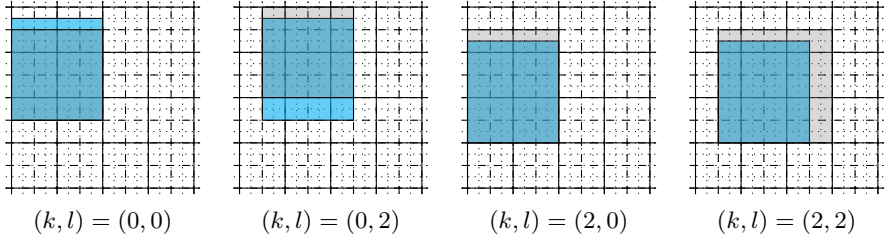
	\caption{The modification of patterns in the sets $Q_{2+k, 1+l}(T, (4n+1)\times (4n))$ leading to a recursion for $c_{2,1}(4n)$. The blue rectangle indicates the starting pattern, and the gray the concluding one, compare \eqref{eq:recc210}.}
	\label{fig:recc210}
\end{figure}

\begin{equation}
\label{eq:reca111}%
\text{\footnotesize$
\renewcommand{\arraystretch}{1.3}
%
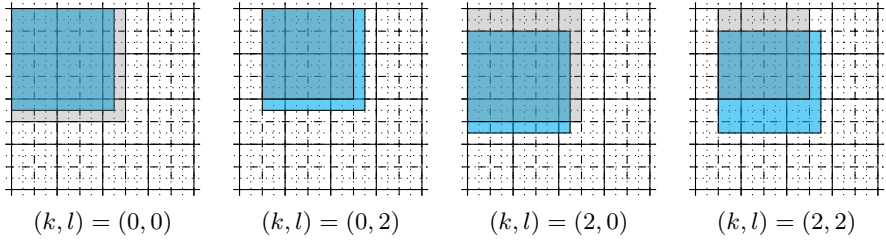
	\caption{The modification of patterns in the sets $Q_{1+k, 1+l}(T, (4n+1)\times (4n+1))$ leading to a recursion for $a_{1,1}(4n+1)$. The blue rectangle indicates the starting pattern, and the gray the concluding one, compare \eqref{eq:reca111}.}
	\label{fig:reca111}
\end{figure}

\begin{equation}
\label{eq:reca211}%
\text{\footnotesize$
\renewcommand{\arraystretch}{1.3}
%
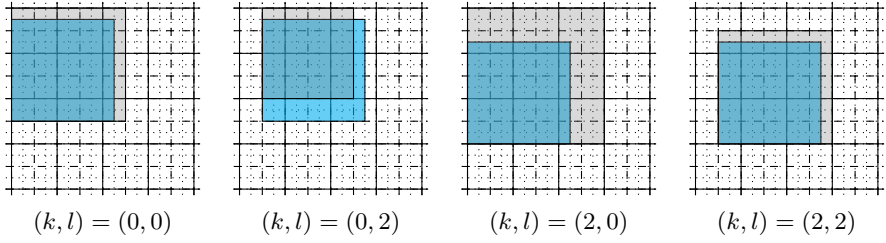
	\caption{The modification of patterns in the sets $Q_{2+k, 1+l}(T, (4n+1)\times (4n+1))$ leading to a recursion for $a_{2,1}(4n+1)$. The blue rectangle indicates the starting pattern, and the gray the concluding one, compare \eqref{eq:reca211}.}
	\label{fig:reca211}
\end{figure}

\begin{equation}
\label{eq:reca221}%
\text{\footnotesize$
\renewcommand{\arraystretch}{1.3}
%
	\caption{The modification of patterns in the sets $Q_{2+k, 2+l}(T, (4n+1)\times (4n+1))$ leading to a recursion for $a_{2,2}(4n+1)$. The blue rectangle indicates the starting pattern, and the gray the concluding one, compare \eqref{eq:reca221}.}
	\label{fig:reca221}
\end{figure}

\begin{equation}
\label{eq:recb121}%
\text{\footnotesize$
\renewcommand{\arraystretch}{1.3}
%
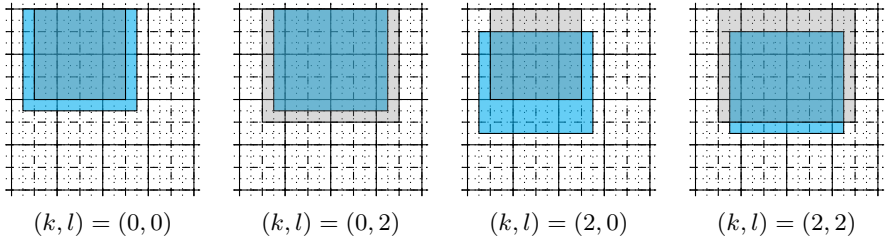
	\caption{The modification of patterns in the sets $Q_{1+k, 2+l}(T, (4n+1)\times (4n+2))$ leading to a recursion for $b_{1,2}(4n+1)$. The blue rectangle indicates the starting pattern, and the gray the concluding one, compare \eqref{eq:recb121}.}
	\label{fig:recb121}
\end{figure}

\begin{equation}
\label{eq:recc221}%
\text{\footnotesize$
\renewcommand{\arraystretch}{1.3}
%
	\caption{The modification of patterns in the sets $Q_{2+k, 2+l}(T, (4n+2)\times (4n+1))$ leading to a recursion for $c_{2,2}(4n+1)$. The blue rectangle indicates the starting pattern, and the gray the concluding one, compare \eqref{eq:recc221}.}
	\label{fig:recc221}
\end{figure}

\begin{equation}
\label{eq:reca122}%
\text{\footnotesize$
\renewcommand{\arraystretch}{1.3}
%
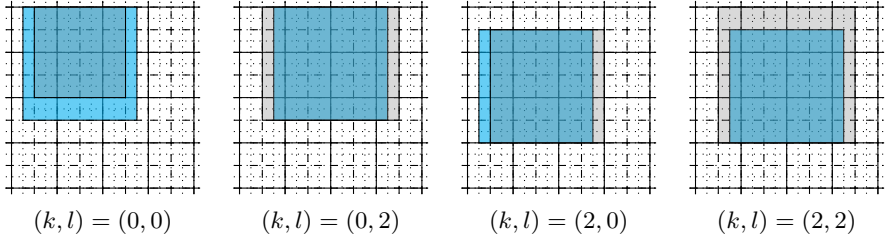
	\caption{The modification of patterns in the sets $Q_{1+k, 2+l}(T, (4n+2)\times (4n+2))$ leading to a recursion for $a_{1,2}(4n+2)$. The blue rectangle indicates the starting pattern, and the gray the concluding one, compare \eqref{eq:reca122}.}
	\label{fig:reca122}
\end{figure}

\begin{equation}
\label{eq:reca212}%
\text{\footnotesize$
\renewcommand{\arraystretch}{1.3}
%
	\caption{The modification of patterns in the sets $Q_{2+k, 1+l}(T, (4n+2)\times (4n+2))$ leading to a recursion for $a_{2,1}(4n+2)$. The blue rectangle indicates the starting pattern, and the gray the concluding one, compare \eqref{eq:reca212}.}
	\label{fig:reca212}
\end{figure}

\begin{equation}
\label{eq:reca222}%
\text{\footnotesize$
\renewcommand{\arraystretch}{1.3}
%
	\caption{The modification of patterns in the sets $Q_{2+k, 2+l}(T, (4n+2)\times (4n+2))$ leading to a recursion for $a_{2,2}(4n+2)$. The blue rectangle indicates the starting pattern, and the gray the concluding one, compare \eqref{eq:reca222}.}
	\label{fig:reca222}
\end{figure}

\begin{equation}
\label{eq:recb112}%
\text{\footnotesize$
\renewcommand{\arraystretch}{1.3}
%
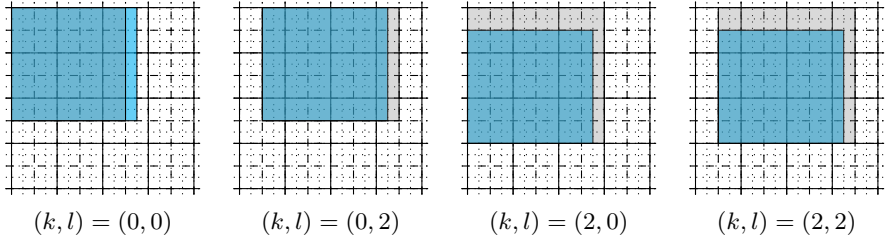
	\caption{The modification of patterns in the sets $Q_{1+k, 1+l}(T, (4n+2)\times (4n+3))$ leading to a recursion for $b_{1,1}(4n+2)$. The blue rectangle indicates the starting pattern, and the gray the concluding one, compare \eqref{eq:recb112}.}
	\label{fig:recb112}
\end{figure}

\begin{equation}
\label{eq:recc212}%
\text{\footnotesize$
\renewcommand{\arraystretch}{1.3}
%
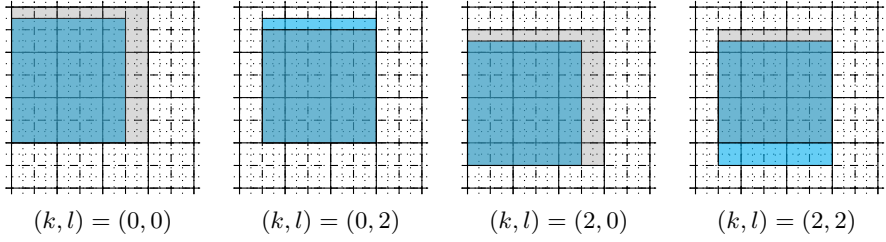
	\caption{The modification of patterns in the sets $Q_{2+k, 1+l}(T, (4n+3)\times (4n+2))$ leading to a recursion for $c_{2,1}(4n+2)$. The blue rectangle indicates the starting pattern, and the gray the concluding one, compare \eqref{eq:recc212}.}
	\label{fig:recc212}
\end{figure}

\begin{equation}
\label{eq:reca113}%
\text{\footnotesize$
\renewcommand{\arraystretch}{1.3}
%
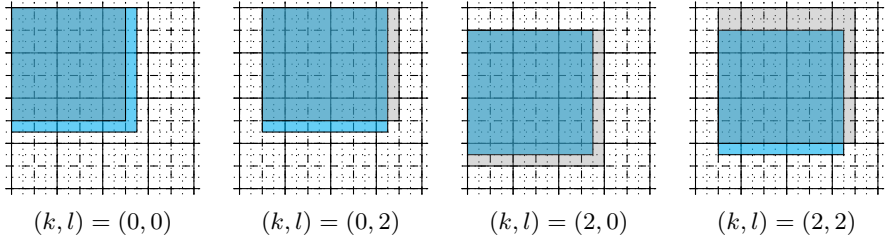
	\caption{The modification of patterns in the sets $Q_{1+k, 1+l}(T, (4n+3)\times (4n+3))$ leading to a recursion for $a_{1,1}(4n+3)$. The blue rectangle indicates the starting pattern, and the gray the concluding one, compare \eqref{eq:reca113}.}
	\label{fig:reca113}
\end{figure}

\begin{equation}
\label{eq:reca213}%
\text{\footnotesize$
\renewcommand{\arraystretch}{1.3}
%
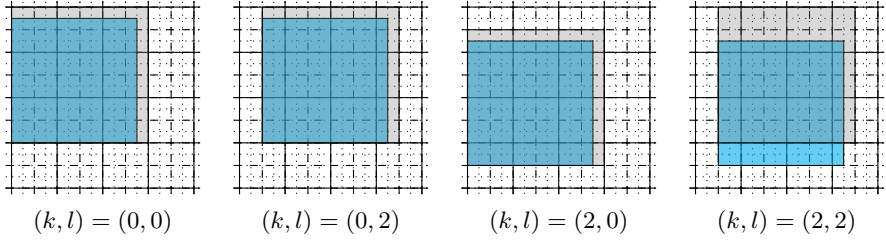
	\caption{The modification of patterns in the sets $Q_{2+k, 1+l}(T, (4n+3)\times (4n+3))$ leading to a recursion for $a_{2,1}(4n+3)$. The blue rectangle indicates the starting pattern, and the gray the concluding one, compare \eqref{eq:reca213}.}
	\label{fig:reca213}
\end{figure}

\begin{equation}
\label{eq:reca223}%
\text{\footnotesize$
\renewcommand{\arraystretch}{1.3}
%
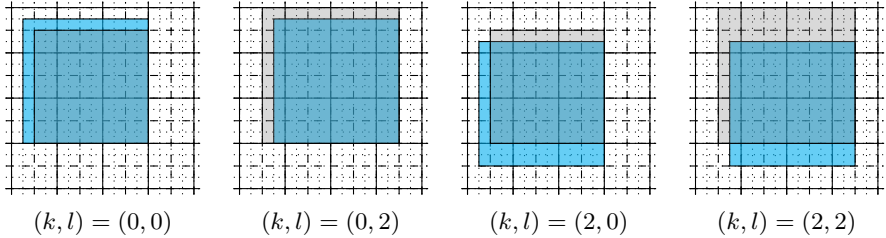
	\caption{The modification of patterns in the sets $Q_{2+k, 2+l}(T, (4n+3)\times (4n+3))$ leading to a recursion for $a_{2,2}(4n+3)$. The blue rectangle indicates the starting pattern, and the gray the concluding one, compare \eqref{eq:reca223}.}
	\label{fig:reca223}
\end{figure}

\begin{equation}
\label{eq:recb123}%
\text{\footnotesize$
\renewcommand{\arraystretch}{1.3}
%
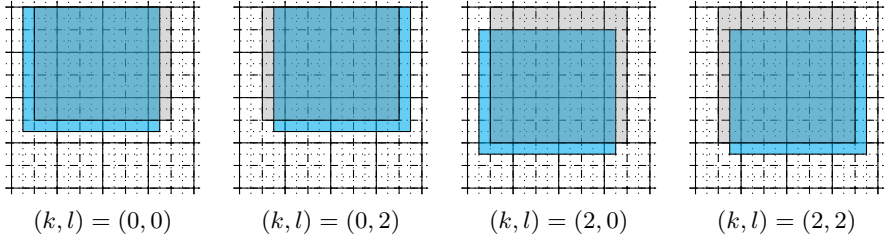
	\caption{The modification of patterns in the sets $Q_{1+k, 2+l}(T, (4n+3)\times (4n+4))$ leading to a recursion for $b_{1,2}(4n+3)$. The blue rectangle indicates the starting pattern, and the gray the concluding one, compare \eqref{eq:recb123}.}
	\label{fig:recb123}
\end{figure}

\begin{equation}
\label{eq:recb223}%
\text{\footnotesize$
\renewcommand{\arraystretch}{1.3}
%
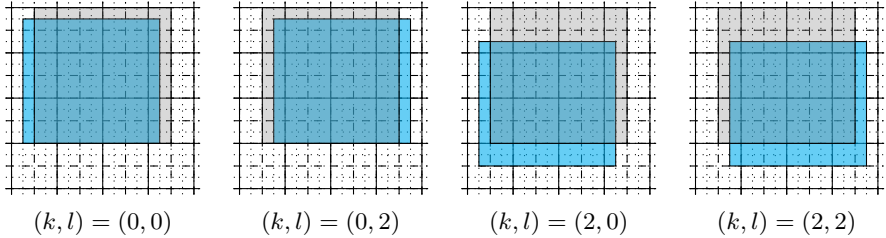
	\caption{The modification of patterns in the sets $Q_{2+k, 2+l}(T, (4n+3)\times (4n+4))$ leading to a recursion for $b_{2,2}(4n+3)$. The blue rectangle indicates the starting pattern, and the gray the concluding one, compare \eqref{eq:recb223}.}
	\label{fig:recb223}
\end{figure}

\begin{equation}
\label{eq:recc213}%
\text{\footnotesize$
\renewcommand{\arraystretch}{1.3}
%
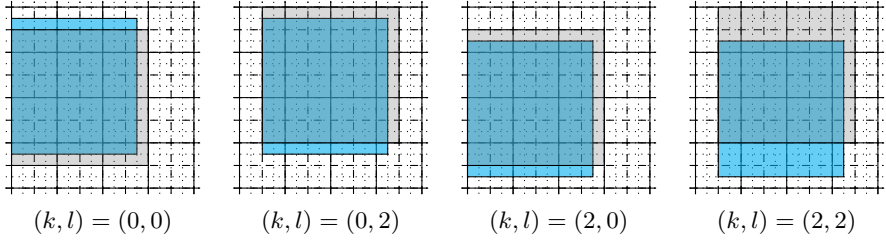
	\caption{The modification of patterns in the sets $Q_{2+k, 1+l}(T, (4n+4)\times (4n+3))$ leading to a recursion for $c_{2,1}(4n+3)$. The blue rectangle indicates the starting pattern, and the gray the concluding one, compare \eqref{eq:recc213}.}
	\label{fig:recc213}
\end{figure}

\begin{equation}
\label{eq:recc223}%
\text{\footnotesize$
\renewcommand{\arraystretch}{1.3}
%
	\caption{The modification of patterns in the sets $Q_{2+k, 2+l}(T, (4n+4)\times (4n+3))$ leading to a recursion for $c_{2,2}(4n+3)$. The blue rectangle indicates the starting pattern, and the gray the concluding one, compare \eqref{eq:recc223}.}
	\label{fig:recc223}
\end{figure}

By the help of all the above cases we can summarize \eqref{eq:reca110} -- \eqref{eq:recc223} into the following system of recursions.

\begin{equation}
\label{eq:recsysA}
\left\{
\text{\footnotesize
$\begin{array}{l*{4}{@{\ }c@{\ }l}}
a_{1,1}(4n)    &=& a_{1,1}(2n)    &+& a_{1,2}(2n)    &+& a_{2,1}(2n)    &+& a_{2,2}(2n)    \\
a_{1,2}(4n)    &=& a_{1,2}(2n)    &+& a_{1,2}(2n)    &+& a_{1,2}(2n)    &+& a_{1,2}(2n)    \\
a_{2,1}(4n)    &=& a_{2,1}(2n)    &+& a_{1,2}(2n)    &+& a_{2,1}(2n)    &+& a_{1,2}(2n)    \\
a_{2,2}(4n)    &=& a_{1,2}(2n)    &+& a_{1,2}(2n)    &+& a_{1,2}(2n)    &+& a_{1,2}(2n)    \\
\\
a_{1,1}(4n+1)  &=& a_{1,1}(2n+1)  &+& a_{1,2}(2n)    &+& a_{1,1}(2n+1)  &+& a_{1,2}(2n)    \\
a_{1,2}(4n+1)  &=& a_{2,2}(4n)    \\
a_{2,1}(4n+1)  &=& a_{1,1}(2n+1)  &+& a_{1,2}(2n)    &+& a_{1,1}(2n+2)  &+& a_{2,2}(2n+1)  \\
a_{2,2}(4n+1)  &=& a_{1,2}(2n)    &+& a_{1,2}(2n)    &+& a_{2,2}(2n+1)  &+& a_{2,2}(2n+1)  \\
\\
a_{1,1}(4n+2)  &=& a_{2,1}(4n+1)  \\
a_{1,2}(4n+2)  &=& a_{1,2}(2n)    &+& b_{1,2}(2n+1)  &+& a_{2,2}(2n+1)  &+& a_{1,2}(2n+2)  \\
a_{2,1}(4n+2)  &=& a_{1,1}(2n+1)  &+& a_{1,2}(2n)    &+& a_{2,1}(2n+2)  &+& a_{2,2}(2n+1)  \\
a_{2,2}(4n+2)  &=& a_{1,2}(2n)    &+& b_{1,2}(2n+1)  &+& a_{2,2}(2n+1)  &+& a_{1,2}(2n+2)  \\
\\
a_{1,1}(4n+3)  &=& a_{1,1}(2n+1)  &+& b_{1,2}(2n+1)  &+& a_{2,1}(2n+2)  &+& a_{1,2}(2n+2)  \\
a_{1,2}(4n+3)  &=& a_{2,2}(4n+2)  \\
a_{2,1}(4n+3)  &=& a_{1,1}(2n+2)  &+& a_{1,2}(2n+2)  &+& a_{2,1}(2n+2)  &+& a_{1,2}(2n+2)  \\
a_{2,2}(4n+3)  &=& a_{2,2}(2n+1)  &+& a_{1,2}(2n+2)  &+& a_{2,2}(2n+1)  &+& a_{1,2}(2n+2)  \\
\end{array}
$}\right.
\end{equation}

\begin{equation}
\label{eq:recsysB}
\left\{
\text{\footnotesize
$\begin{array}{l*{4}{@{\ }c@{\ }l}}
b_{1,1}(4n)    &=& b_{1,1}(2n)    &+& a_{1,2}(2n)    &+& a_{1,1}(2n+1)  &+& a_{1,2}(2n)    \\
b_{1,2}(4n)    &=& a_{1,2}(4n)    \\
b_{2,1}(4n)    &=& a_{1,1}(4n+1)  \\
b_{2,2}(4n)    &=& a_{2,2}(4n)    \\
\\
b_{1,1}(4n+1)  &=& a_{1,1}(4n+1)  \\
b_{1,2}(4n+1)  &=& a_{1,2}(2n)    &+& b_{1,2}(2n+1)  &+& a_{1,2}(2n)    &+& b_{1,2}(2n+1)  \\
b_{2,1}(4n+1)  &=& a_{2,1}(4n+1)  \\
b_{2,2}(4n+1)  &=& a_{1,2}(4n+2)  \\
\\
b_{1,1}(4n+2)  &=& a_{1,1}(2n+1)  &+& b_{1,2}(2n+1)  &+& a_{1,1}(2n+2)  &+& a_{1,2}(2n+2)  \\
b_{1,2}(4n+2)  &=& a_{1,2}(4n+2)  \\
b_{2,1}(4n+2)  &=& a_{1,1}(4n+3)  \\
b_{2,2}(4n+2)  &=& a_{2,2}(4n+2)  \\
\\
b_{1,1}(4n+3)  &=& a_{1,1}(4n+3)  \\
b_{1,2}(4n+3)  &=& b_{1,2}(2n+1)  &+& b_{1,2}(2n+1)  &+& a_{1,2}(2n+2)  &+& a_{1,2}(2n+2)  \\
b_{2,1}(4n+3)  &=& a_{2,1}(4n+3)  \\
b_{2,2}(4n+3)  &=& a_{1,2}(2n+2)  &+& a_{1,2}(2n+2)  &+& a_{1,2}(2n+2)  &+& a_{1,2}(2n+2)  \\
\end{array}
$}\right.
\end{equation}

\begin{equation}
\label{eq:recsysC}
\left\{
\text{\footnotesize
$\begin{array}{l*{4}{@{\ }c@{\ }l}}
c_{1,1}(4n)    &=& a_{2,1}(4n)    \\
c_{1,2}(4n)    &=& a_{2,2}(4n)    \\
c_{2,1}(4n)    &=& a_{2,1}(2n)    &+& a_{1,2}(2n)    &+& c_{2,1}(2n)    &+& a_{2,2}(2n+1)  \\
c_{2,2}(4n)    &=& a_{2,2}(4n+1)  \\
\\
c_{1,1}(4n+1)  &=& a_{2,1}(4n+1)  \\
c_{1,2}(4n+1)  &=& a_{2,2}(4n+1)  \\
c_{2,1}(4n+1)  &=& a_{2,1}(4n+2)  \\
c_{2,2}(4n+1)  &=& a_{1,2}(2n)    &+& a_{1,2}(2n)    &+& a_{2,2}(2n+1)  &+& a_{2,2}(2n+1)  \\
\\
c_{1,1}(4n+2)  &=& a_{2,1}(4n+2)  \\
c_{1,2}(4n+2)  &=& a_{2,2}(4n+2)  \\
c_{2,1}(4n+2)  &=& a_{1,1}(2n+2)  &+& a_{2,2}(2n+1)  &+& a_{2,1}(2n+2)  &+& a_{2,2}(2n+1)  \\
c_{2,2}(4n+2)  &=& a_{2,2}(4n+3)  \\
\\
c_{1,1}(4n+3)  &=& a_{2,1}(4n+3)  \\
c_{1,2}(4n+3)  &=& a_{2,2}(4n+3)  \\
c_{2,1}(4n+3)  &=& a_{2,1}(2n+2)  &+& a_{1,2}(2n+2)  &+& a_{2,1}(2n+2)  &+& a_{1,2}(2n+2)  \\
c_{2,2}(4n+3)  &=& a_{2,2}(2n+1)  &+& a_{1,2}(2n+2)  &+& a_{2,2}(2n+1)  &+& a_{1,2}(2n+2)  \\
\end{array}
$}\right.
\end{equation}

The remaining part needed to complete the systems given \eqref{eq:recsysA}, \eqref{eq:recsysB}, and \eqref{eq:recsysC} is to combine them with the initial values; as given in Table~\ref{table:initialvalues}.

\begin{table}[ht]
\footnotesize
\centering
\begin{tabular}{c*{11}{r}}
\toprule
$n$ & 
\multicolumn{1}{p{7mm}}{\hfill 1} & 
\multicolumn{1}{p{7mm}}{\hfill 2} & 
\multicolumn{1}{p{7mm}}{\hfill 3} & 
\multicolumn{1}{p{7mm}}{\hfill 4} & 
\multicolumn{1}{p{7mm}}{\hfill 5} & 
\multicolumn{1}{p{7mm}}{\hfill 6} & 
\multicolumn{1}{p{7mm}}{\hfill 7} & 
\multicolumn{1}{p{7mm}}{\hfill 8} &
\multicolumn{1}{p{7mm}}{\hfill 9} & 
\multicolumn{1}{p{7mm}}{\hfill10} \\ 
\midrule
$a_{1,1}(n)$ &      4 &     16 &     48 &     76 &    136 &    184 &    256 &    316 &    432 &    520 \\
$a_{1,2}(n)$ &      4 &     20 &     20 &     80 &     80 &    188 &    188 &    320 &    320 &    524 \\
$a_{2,1}(n)$ &      4 &     20 &     76 &     80 &    184 &    188 &    316 &    320 &    520 &    524 \\
$a_{2,2}(n)$ &      4 &     20 &     40 &     80 &    120 &    188 &    240 &    320 &    400 &    524 \\
\midrule
$b_{1,1}(n)$ &      8 &     40 &     48 &    128 &    136 &    252 &    256 &    424 &    432 &    644 \\
$b_{1,2}(n)$ &      8 &     20 &     48 &     80 &    136 &    188 &    256 &    320 &    432 &    524 \\
$b_{2,1}(n)$ &      8 &     48 &     76 &    136 &    184 &    256 &    316 &    432 &    520 &    648 \\
$b_{2,2}(n)$ &     16 &     20 &     80 &     80 &    188 &    188 &    320 &    320 &    524 &    524 \\
\midrule
$c_{1,1}(n)$ &      8 &     20 &     76 &     80 &    184 &    188 &    316 &    320 &    520 &    524 \\
$c_{1,2}(n)$ &      8 &     20 &     40 &     80 &    120 &    188 &    240 &    320 &    400 &    524 \\
$c_{2,1}(n)$ &     12 &     36 &     80 &    116 &    188 &    236 &    320 &    396 &    524 &    612 \\
$c_{2,2}(n)$ &      8 &     40 &     40 &    120 &    120 &    240 &    240 &    400 &    400 &    616 \\
\midrule
$A_{n}$      &      4 &     68 &    184 &    316 &    520 &    748 &   1000 &   1276 &   1672 &   2092 \\
\bottomrule
\end{tabular}
\caption{Initial terms for $a$, $b$, $c$, and $A$.}
\label{table:initialvalues}
\end{table}

\section{Proof of Main Theorem}
\label{sec:proofofmainthm}

This section is devoted to give a simplified recursion for the number of patterns of given size, based on the system in \eqref{eq:recsysA} \eqref{eq:recsysB}, and \eqref{eq:recsysC}, and then use this simplification to prove Theorem~\ref{thm:main}.

We start by expanding the list of equalities given in Lemma~\ref{lemma:first_abc_equalties}.

\begin{lemma}
\label{lemma:second_abc_equalties}
Let $n\geq1$. Then 
\begin{equation}
\label{eq:second_abc_equalties}
\renewcommand{\arraystretch}{1.3}
\begin{array}{r@{\ }l}
	  a_{1,2}(2n) &= a_{2,1}(2n) = a_{2,2}(2n), \\
	a_{1,1}(2n+1) &= b_{1,2}(2n+1).
\end{array}
\end{equation}
\end{lemma}

\begin{proof}
The initial cases, $n = 1,2,3,4$ are directly seen from Table~\ref{table:initialvalues}. To prove that the equalities in \eqref{eq:second_abc_equalties} hold,  we consider two cases depending on the parity of $n$, that is, 
\begin{equation}
\label{eq:assumption_second_abc_equalties}
\renewcommand{\arraystretch}{1.3}
\begin{array}{r@{\ }c@{\ }l@{\ }c@{\ }l}
	a_{1,2}(4k)   &=& a_{2,1}(4k)   &=& a_{2,2}(4k), \\
	a_{1,2}(4k+2) &=& a_{2,1}(4k+2) &=& a_{2,2}(4k+2), \\
	a_{1,1}(4k+1) &=& b_{1,2}(4k+1), \\
	a_{1,1}(4k+3) &=& b_{1,2}(4k+3), \\
\end{array}
\end{equation}
for $k\geq1$. Here we give a proof by induction on $k$. The basis cases are, as just mentioned, clear. Assume for induction that the equalities in \eqref{eq:assumption_second_abc_equalties} hold for $1\leq k < p$. Then, for the induction step, Lemma~\ref{lemma:first_abc_equalties}, the recursion from \eqref{eq:recsysA} \eqref{eq:recsysB}, and \eqref{eq:recsysC}, and the induction assumption imply 
\begin{itemize}
\item
$\begin{aligned}[t]	
	a_{1,2}(4p) - a_{2,1}(4p) = 2a_{1,2}(2p) - 2a_{2,1}(2p) = 0,
\end{aligned}$
\item
$\begin{aligned}[t]	
	a_{1,2}(4p) - a_{2,2}(4p) = 4a_{1,2}(2p) - 4a_{1,2}(2p) = 0,
\end{aligned}$
\item
$\begin{aligned}[t]	
	a_{1,1}(4p+1) - b_{1,2}(4p+1) = 2a_{1,1}(2p+1) - 2b_{1,2}(2p+1) = 0,  
\end{aligned}$
\item
$\begin{aligned}[t]	
	a_{1,2}(4p+2) - a_{2,1}(4p+2) = &
 \phantom{+} \  b_{1,2}(2p+1) + a_{1,2}(2p+2) \\
&-a_{1,1}(2p+1) - a_{2,1}(2p+2)   = 0,
\end{aligned}$
\item
$\begin{aligned}[t]	
	a_{1,2}(4p+2) - a_{2,2}(4p+2) =  
  a_{1,1}(2p) - a_{1,2}(2p)  = 0,
\end{aligned}$
\item
$\begin{aligned}[t]	
	a_{1,1}(4p+3) - b_{1,2}(4p+3) = a_{1,1}(2p+1) -b_{1,2}(2p+1) = 0,
\end{aligned}$
\end{itemize}
which completes the proof. 
\end{proof}

\begin{lemma}
\label{lemma:plus4}
Let $n\geq 1$. Then 
\begin{equation}
\label{eq:plus4}
\begin{split}
	a_{1,1}(2n)   + 4 &= a_{1,2}(2n), \\
	a_{2,1}(2n+1) + 4 &= b_{2,2}(2n+1).
\end{split}
\end{equation}
\end{lemma}

\begin{proof}
We split the equalities in \eqref{eq:plus4} into the cases depending on the parity of $n$. That is, we look at 
\begin{equation}
\label{eq:proof_plus4_cases}
\renewcommand{\arraystretch}{1.3}
\begin{array}{r@{\ }c@{\ }l@{\ }}
	a_{1,1}(4k)   + 4 &=& a_{1,2}(4k), \\
	a_{1,1}(4k+2) + 4 &=& a_{1,2}(4k+2), \\
	a_{2,1}(4k+1) + 4 &=& b_{2,2}(4k+1), \\
	a_{2,1}(4k+3) + 4 &=& b_{2,2}(4k+3), \\
\end{array}
\end{equation}
for $k\geq1$. The case $n=1$ and $k=1,2$ are direct from Table~\ref{table:initialvalues}.
We prove the equalities in \eqref{eq:proof_plus4_cases} by induction on $k$. The basis cases are, as just mentioned clear. Assume for induction that the equalities in \eqref{eq:proof_plus4_cases} hold for $1\leq k < p$. Then, for the induction step, Lemma~\ref{lemma:first_abc_equalties}, Lemma~\ref{lemma:second_abc_equalties}, and the induction assumption, give
\begin{itemize}
\item
$\begin{aligned}[t]	
	a_{1,2}(4p) - a_{1,1}(4p)
	&= a_{1,2}(2n) + a_{2,2}(2n) - a_{1,1}(2n) -	a_{2,1}(2n) \\
		&= a_{1,2}(2p) - a_{1,1}(2p) \\
		&= 4,
\end{aligned}$

\item
$\begin{aligned}[t]	
	a_{1,2}(4p+2) - a_{1,1}(4p+2)
		&= b_{1,2}(2p+1) + b_{2,2}(2p+1)  \\
		&\phantom{=} - a_{1,1}(2p+1) - a_{2,1}(2p+1) \\
		&= b_{2,2}(2p+1) - a_{2,1}(2p+1) \\
		&= 4,
\end{aligned}$

\item
$\begin{aligned}[t]	
	b_{2,2}(4p+1) - a_{2,1}(4p+1) &= a_{1,2}(4p+2) - a_{1,1}(4p+2) = 4,
\end{aligned}$

\item
$\begin{aligned}[t]	
b_{2,2}(4p+3) - a_{2,1}(4p+3) 
	&= a_{1,2}(2p+2) + a_{2,2}(2p+2)  \\
	&\phantom{=} - a_{2,1}(2p+1) - a_{2,1}(2p+2) \\
	&= a_{1,2}(2p+2) - a_{2,1}(2p+1) \\
	&= b_{2,2}(2p+1) - a_{2,1}(2p+1) \\
	&= 4,
\end{aligned}$
\end{itemize}
which completes the proof. 
\end{proof}

\begin{lemma}
\label{lemma:recAnEven}
Let $n\geq3$. Then $A_{2n} = 4A_{n} + 12.$
\end{lemma}

\begin{proof}
By Lemma~\ref{lemma:second_abc_equalties} and Lemma~\ref{lemma:plus4} we have
\[
	A_{2n} = a_{1,1}(2n)  + 3 a_{1,2}(2n) = 4 a_{1,1}(2n) + 12  = 4 A_n + 12,
\]
which concludes the proof.
\end{proof}

\begin{lemma}
\label{lemma:recAnOdd}
Let $n\geq3$. Then 
\begin{equation}
\label{eq:recAnOdd}
	A_{2n+1} = 2A_{n} + 2A_{n+1}.
\end{equation}
\end{lemma}

\begin{proof}
We split the equality in \eqref{eq:recAnOdd} into cases depending on the parity of $n$, that is, 
\begin{equation}
\label{eq:proof_recAnOdd_cases}
\renewcommand{\arraystretch}{1.3}
\begin{array}{l@{\ }c@{\ }l@{\ }}
	A_{4k+1} &=& 2A_{2k}   + 2A_{2k+1}, \\
	A_{4k+3} &=& 2A_{2k+1} + 2A_{2k+2}.
\end{array}
\end{equation}
We prove the equalities in \eqref{eq:proof_recAnOdd_cases} by induction on $k$. The basis cases are direct from Table~\ref{table:initialvalues}. Assume for induction that the equalities in \eqref{eq:proof_recAnOdd_cases} hold for $1\leq k < p$. Then, for the induction step, Lemma~\ref{lemma:first_abc_equalties}, Lemma~\ref{lemma:second_abc_equalties}, Lemma~\ref{lemma:recAnEven}, and the induction assumption give 
\begin{align*}
A_{4p+1} - 2A_{2p} - 2A_{2p+1}  
	&= a_{1,1}(4p+2) -2a_{1,1}(2p) - 2a_{2,1}(2p+1)\\
	&= a_{1,1}(4p+2) -2a_{1,1}(2p) - 2a_{1,1}(2p+2)\\
	&= A_{2p+1} - 2A_{p} - 2A_{p+1}  \\
	&= 0.
\end{align*}
In a similar way, with also the help of Lemma~\ref{lemma:plus4}, we obtain
\begin{align*}
A_{4p+3} -  & \, 2A_{2p+1} - 2A_{2p+2}  \\
&= a_{1,1}(2n+1) + a_{2,2}(2p+1) + 2a_{1,2}(2p+2)\\
&\quad - a_{1,2}(2p+1) - a_{2,1}(2p+1) - 2A_{p+1}\\
&= a_{1,1}(2p+1) + a_{2,2}(2p+1) + 2a_{1,2}(2p+2)\\
&\quad + a_{1,2}(2p+1) + a_{2,1}(2p+1) \\
&\quad - 2a_{1,2}(2p+1) - 2a_{2,1}(2p+1) - 2A_{p+1}\\
&= A_{2p+1} - 2A_{p+1}\\
&\quad + 2a_{1,2}(2p+2) - 2a_{1,2}(2p+1) - 2a_{2,1}(2p+1) \\
&= A_{2p+1} - 2A_{p+1} + 8 - 2a_{1,2}(2p+1) \\
&= A_{2p+1} - 2A_{p+1} - 2a_{1,1}(2p) \\
&= A_{2p+1} - 2A_{p+1} - 2A_{p} \\
&= 0,
\end{align*}
which completes the induction. 
\end{proof}

To conclude; Lemma~\ref{lemma:recAnEven} and Lemma~\ref{lemma:recAnOdd} give us the system of recursion
\begin{equation}
\label{eq:Aeqsys}
	\left\{
	\begin{aligned}
		A_{2n}   &= 4 A_{n} +12 , \\
		A_{2n+1} &= 2 A_{n} + 2 A_{n+1}, 
	\end{aligned}
\right.
\end{equation}
for $n\geq3$, and with the initial conditions $A_1 = 4$, $A_2 = 68$, $A_3 = 184$, $A_4 =316$, and $A_5 = 520$ from Table~\ref{table:initialvalues}. What remains is to prove that \eqref{eq:AnSolution} is a solution to \eqref{eq:Aeqsys}. 

\begin{proof}[Proof of Theorem~\ref{thm:main}]
For the even index case of \eqref{eq:Aeqsys} we use \eqref{eq:AnSolution} to obtain
\begin{align*}
A_{2n} 
	&= 12(2n)^2 + 24 \cdot (2n) \cdot 2^{\lfloor \log_2 (2n-1) \rfloor  } - 16\cdot 4^{\lfloor \log_2 (2n-1) \rfloor  } - 4 \\
	&= 4\cdot 12 \cdot n^2 + 2\cdot 24n \cdot 2^{\lfloor \log_2 2(n-\frac{1}{2}) \rfloor} - 16\cdot 4^{\lfloor \log_2 2(n-\frac{1}{2}) \rfloor} - 4 \\
	&= 4\cdot 12 \cdot n^2 + 4\cdot 24n \cdot 2^{\lfloor \log_2 (n-\frac{1}{2}) \rfloor} - 4\cdot 16\cdot 4^{\lfloor \log_2 (n-\frac{1}{2}) \rfloor} - 4\cdot4 + 12\\
	&= 4 \cdot A_n + 12,
\end{align*}
since $\lfloor \log_2 (n-\frac{1}{2}) \rfloor = \lfloor \log_2 (n-1) \rfloor$, when $n\geq2$ and is an integer. For the odd index case of \eqref{eq:Aeqsys}, we start from the expression 
\begin{align}
A_{2n+1} - 2\cdot A_n - 2 \cdot A_{n+1} 
	&= 48 n\cdot 2^{\lfloor\log_2 n \rfloor} - 48n\cdot 2^{\lfloor\log_2(n-1)\rfloor} \nonumber \\
	&\quad - 32\cdot 4^{\lfloor\log_2 n\rfloor} + 32\cdot 4^{\lfloor\log_2(n-1)\rfloor}.
\label{eq:SolvingAnOdd}
\end{align}
If  $n \neq 2^k$ then $\lfloor \log_2 n \rfloor = \lfloor \log_2 (n-1) \rfloor$ and the terms in \eqref{eq:SolvingAnOdd} cancels out. If $n = 2^k$ then 
$\lfloor \log_2 n \rfloor = 1+\lfloor \log_2 (n-1) \rfloor$ and we have 
\begin{align*}
A_{2^{k+1}+1} 
	&- 2\cdot A_{2^k} - 2 \cdot A_{2^{k}+1} \\ 
	&= 48\cdot 2^k \cdot 2^k - 48\cdot 2^k \cdot 2^{k-1} - 32\cdot 4^k + 32\cdot 4^{k-1} \\
	&= 48\cdot 2^{2k-1} \cdot (2 - 1) - 32\cdot 4^{k-1}\cdot(4-1) \\
	&= 24\cdot 4^{k} - 3\cdot 8\cdot 4^k  \\
	&= 0,
\end{align*}
which concludes to proof.
\end{proof}

\section{Outlook}

Our technique presented in this paper is fairly general and can be applied to find the pattern complexity of other tilings that are obtained via a substitution rule, such as the Table (or Domino) tiling (see \cite{FHG}), the Squiral tiling \cite{FHG}, and the structure generated by the Ulam-Warburton cellular automaton \cite{appelgate2010, singmaster, ulam1962, warburton2019, warburton2002}, just to mention few examples. However, things are more subtle; it is easy to find substitution rules where the method of this paper does not seem to work, i.e the Chair tiling, see \cite{FHG}. A different technique is applied by Galanov \cite{galanov} who looks at the pattern complexity of the Robinson tiling \cite{robinson}.

Our work here depends heavily on the extension properties, as discussed in section~\ref{sec:extensions}. One question is to find necessary and sufficient conditions on the substitution rule in order to be able to calculate the pattern complexity in our way, and to classify different types of complexity functions that might appear.

\noindent\rule{10em}{0.5pt}

\texttt{\small johannilsson514@gmail.com}
\end{document}